\documentclass[reqno]{amsart}

\usepackage[usenames,dvipsnames]{pstricks}
\usepackage{epsfig}

\usepackage{color}
\date{\today}

\usepackage[latin1]{inputenc}
\usepackage[T1]{fontenc}
\usepackage{amsmath, amsthm}
\usepackage[english]{babel}
\usepackage{amssymb}
\usepackage{latexsym}
\usepackage{graphicx}
\usepackage[all]{xy}

\newtheorem{theorem}{Theorem}[section]
\newtheorem{lemma}[theorem]{Lemma}
\newtheorem{corollary}[theorem]{Corollary}
\newtheorem{proposition}[theorem]{Proposition}
\theoremstyle{definition}
\newtheorem{definition}[theorem]{Definition}
\newtheorem{example}[theorem]{Example}

\theoremstyle{remark}
\newtheorem{remark}[theorem]{Remark}

\newcommand{\ot}{\otimes}
\newcommand{\co}{\circ}

\begin{document}

\begin{center}

{\huge{\bf Fundamental theorems of Doi-Hopf modules in a nonassociative setting}}

\end{center}

\ \\
\begin{center}
{\bf J.N. ALONSO \'ALVAREZ$^{1}$, J.M. FERN\'ANDEZ VILABOA$^{2}$, R.
GONZ\'{A}LEZ RODR\'{I}GUEZ$^{3}$}
\end{center}

\ \\
\hspace{-0,5cm}$^{1}$ Departamento de Matem\'{a}ticas, Universidad
de Vigo, Campus Universitario Lagoas-Marcosende, E-36280 Vigo, Spain
(e-mail: jnalonso@uvigo.es)
\ \\
\hspace{-0,5cm}$^{2}$ Departamento de \'Alxebra, Universidad de
Santiago de Compostela.  E-15771 Santiago de Compostela, Spain
(e-mail: josemanuel.fernandez@usc.es)
\ \\
\hspace{-0,5cm}$^{3}$ Departamento de Matem\'{a}tica Aplicada II,
Universidad de Vigo, Campus Universitario Lagoas-Mar\-co\-sen\-de, E-36310
Vigo, Spain (e-mail: rgon@dma.uvigo.es)
\ \\

{\bf Abstract} In this paper we introduce the notion of weak non-asssociative  Doi-Hopf module and give the Fundamental Theorem of Hopf modules in this setting. Also we prove that there exists  a categorical equivalence that admits as  particular instances the ones constructed in the literature  for Hopf algebras, weak Hopf algebras, Hopf quasigroups, and weak Hopf quasigroups.

\vspace{0.5cm}



\section{introduction}

Let ${\Bbb F}$ be a field and  ${\mathcal C}={\Bbb F}-Vect$. Let $H$ be a Hopf algebra in ${\mathcal C}$ and  let $B$ be a right $H$-comodule algebra with coaction $\rho_{B}:B\rightarrow B\ot H$,  $\rho_{B}(b)=b_{(0)}\ot b_{(1)}$,  where the unadorned tensor product  is the tensor product over ${\Bbb F}$ and for $\rho_{B}(b)$ we used the Sweedler notation. In \cite{Doi83}, Doi introduced the notion of $(H,B)$-Hopf module, as a generalization of the classical notion of Hopf module, defined by Larson and Sweedler in \cite{Lar-Sweed}, in the following way: Let $M$ be a right $B$-module and a right $H$-comodule. If, for all $m\in M$ and $b\in B$, we write $m.b$ for the action and $\rho_{M}(m)=m_{[0]}\ot m_{[1]}$ for the coaction, we will say that $M$ is an $(H,B)$-Hopf module if the equality  
$$\rho_{M}(m.b)=m_{[0]}.b_{(1)}\ot m_{[1]}b_{(2)}$$
holds, where  $m_{[1]}b_{(2)}$ is the product in $H$ of $m_{[1]}$ and $b_{(2)}$. A morphism between two $(H,B)$-Hopf modules is an ${\Bbb F}$-linear map that is $B$-linear and $H$-colinear. Hopf modules and morphisms of Hopf modules constitute the category of $(H,B)$-Hopf modules  denoted by ${\mathcal M}^{H}_{B}$. If there exists a right $H$-comodule map $h:H\rightarrow B$ which is an algebra map, and $M^{co H}=\{m\in M \;| \; \rho_{M}(m)=m\ot 1_{H}\}$, $B^{co H}=\{b\in B \;| \; \rho_{B}(b)=b\ot 1_{H}\}$ are the subobjects of coinvariants, for any $m\in M^{co H}$ and $b\in B^{co H}$ we have that $m.b\in M^{co H}$ and then $M^{co H}$  is a right $B^{co H}$-module. Using this property  Doi proved in Theorem 3 of \cite{Doi83} that $M$ is isomorphic to  $M^{co H}\ot_{B^{co H}} B$ as $(H,B)$-Hopf modules. Moreover, if $N$ is a right $B^{co H}$-module, the tensor product $N\ot_{B^{co H}} B$, with the action and coaction induced by the product of $B$ and the coproduct of $H$, is an $(H,B)$-Hopf module. This construction is functorial and then  we have a functor, called the induction functor, $F=-\ot_{B^{co H}} B:{\mathcal C}_{B^{co H}}\rightarrow {\mathcal M}^{H}_{B}$. Also, for all $M\in {\mathcal M}^{H}_{B}$, the construction of $M^{co H}$ is functorial and we have a functor of coinvariants $G=(\;\;)^{co H}: {\mathcal M}^{H}_{B}\rightarrow {\mathcal C}_{B^{co H}}$ such that $F\dashv G$. Moreover, $F$ and $G$ induce a categorical equivalence between 
${\mathcal M}^{H}_{B}$ and  the category of right $B^{co H}$-modules. This categorical equivalence was called by Doi and Takeuchi in \cite{Doi-Take}, the strong structure theorem for ${\mathcal M}^{H}_{B}$, and, for $B=H$ and $h=id_{H}$, contains as a particular instance the equivalence derived of the Fundamental Theorem of Hopf modules proved by  Larson and Sweedler (see \cite{Lar-Sweed}, \cite{Sweedler}).

The categorical equivalence of the previous paragraph  remains valid for weak Hopf algebras. For a weak Hopf algebra $H$, B\"{o}hm introduced in \cite{bohm}  the category of Hopf modules, denoted by ${\mathcal M}^{H}_{H}$,  in the same way that in the Hopf algebra setting. If $M\in {\mathcal M}^{H}_{H}$, the subobject of coinvariants is defined by $M^{co H}=\{m\in M \;| \; \rho_{M}(m)=m_{[0]}\ot \Pi_{H}^{L}(m_{[1]})\}$, where $\Pi_{H}^{L}$ is the target morphism associated to $H$.  In \cite{bohm} we can find the weak version of the Fundamental Theorem of Hopf modules, i.e.: For all Hopf module $M$, $M^{co H}\ot_{H_{L}} H$ is isomorphic to  $M$ as Hopf modules, where $H_{L}$ is the image of $\Pi_{H}^{L}$. Moreover, if ${\mathcal C}_{H_{L}}$ is the category of right $H_{L}$-modules, there exist two functors $F=-\ot_{H_{L}} H:{\mathcal C}_{H_{L}}\rightarrow {\mathcal M}^{H}_{H}$ and $G=(\;\;)^{co H}: {\mathcal M}^{H}_{H}\rightarrow {\mathcal C}_{H_{L}}$ such that $F$ is left adjoint of $G$ and they induce a pair of inverse equivalences. Therefore, in the weak setting, ${\mathcal M}^{H}_{H}$ is equivalent to ${\mathcal C}_{H_{L}}$. In this case, the following property is a relevant fact for subsequent generalizations: there is an isomorphism of Hopf modules between the tensor product $M^{co H}\ot_{H_{L}} H$ and $M^{co H}\times H$, where $M^{co H}\times H$ is the image of a suitable idempotent morphism $\nabla_{M}:M^{co H}\otimes H\rightarrow M^{co H}\otimes H$.  Later, in \cite{bohm2}, B\"{o}hm  introduced the notion of weak Doi-Hopf module (or weak $(H,B)$-Doi-Hopf module), associated to a weak Hopf algebra $H$ and a right $H$-comodule algebra $B$, and the category of weak Doi-Hopf modules denoted as in the non-weak setting by ${\mathcal M}^{H}_{B}$.
In 2004, Zhang and Zhu proved that for any weak Doi-Hopf module $M$ (also called by these authors weak $(H,B)$-Doi-Hopf module), if there exists a right $H$-comodule map $h:H\rightarrow B$ which is an algebra map, the objects $M$ and $M^{co H}\ot_{B^{co H}} B$ are isomorphic as $(H,B)$-Doi-Hopf modules. In this case $B^{co H}=\{b\in B \;| \; \rho_{B}(b)=b_{(0)}\ot \Pi_{H}^{L}(b_{(1)})\}$ and, if $B=H$ 
and $h=id_{H}$, they recover the isomorphism constructed by B\"{o}hm  in \cite{bohm}. As in the Hopf setting, it is possible to construct the induction functor $F=-\ot_{B^{co H}} B:{\mathcal C}_{B^{co H}}\rightarrow {\mathcal M}^{H}_{B}$ and the functor of coinvariants $G=(\;\;)^{co H}: {\mathcal M}^{H}_{B}\rightarrow {\mathcal C}_{B^{co H}}$. These functors  satisfy  that $F\dashv G$ and $F$ and $G$ is a pair of inverse equivalences. Therefore, ${\mathcal M}^{H}_{B}$ is equivalent to the category of right $B^{co H}$-modules  (see \cite{Hanna}). 

In the two previous paragraphs we wrote about categorical equivalences for categories of Hopf modules connected to associative algebraic structures like Hopf algebras and weak Hopf algebras. An interesting generalization of Hopf algebras are nonassociative Hopf algebras. As in the quasi-Hopf setting, nonassociative Hopf algebras are not associative, but the lack of this property is compensated in this case by some axioms involving the division operation. The notion of nonassociative Hopf algebra in a category of vector spaces was introduced by P\'erez- Izquierdo  \cite{PI2} with the aim of to construct the universal enveloping algebra for Sabinin algebras, prove a Poincar\'e-Birkhoff-Witt Theorem for Sabinin algebras and give a nonassociative version of the Milnor-Moore theorem. Later, Klim and Majid \cite{Majidesfera}, in order
to understand the structure and relevant properties of the algebraic 7-sphere, introduced the notion of Hopf quasigroup.
Hopf quasigroups  are examples of nonassociative Hopf algebras and in recent years, interesting research about its specific structure and its dual has been developed (\cite{Brz}, \cite{Klim}, \cite{Brz2}, \cite{Brz3}, \cite{ZH}, \cite{ZX}, \cite{FT}, \cite{our1}, \cite{our2}).  Moreover, nonassociative Hopf algebras arise naturally related with other structures
in various nonassociative contexts like, for example quantum quasigroups in the sense of Smith (\cite{S1}, \cite{S2}, \cite{S3}, \cite{S4}).  Nonassociative Hopf algebras include the example of an enveloping algebra $U(L)$ of a Malcev algebra (see \cite{PIS}, \cite{Majidesfera}, \cite{TV}) as well as the notion of the loop algebra RL of a loop L (see \cite{BMP-I12}, \cite{MPIS14}). Then, nonassociative Hopf algebras unify Moufang loops and Malcev algebras, and, more generally, formal loops and Sabinin algebras, in the same way that Hopf algebras unify groups and Lie algebras.

For a of Hopf quasigroup in the sense of  \cite{Majidesfera}, Brzezi\'nski defined in \cite{Brz} the notion of Hopf module obtaining a categorical equivalence as in the associative context. In this case, the main difference  appears in the definition of the category of Hopf modules ${\mathcal M}^{H}_{H}$. Firstly, because the notion of Hopf module reflects the non-associativity of the product defined on $H$. Secondly,  the morphisms are $H$-quasilinear and $H$-colinear  (see Definition 3.4 of \cite{Brz}). In Lemma 3.5 of \cite{Brz}, we can find that, if $M\in {\mathcal M}^{H}_{H}$ and $M^{co H}$ is defined like in the Hopf algebra setting, $M$ is isomorphic to  $M^{co H}\ot H$ as Hopf modules. Therefore the Fundamental Theorem of Hopf modules also holds for Hopf quasigroups. Moreover, there exist two functors $F=-\ot H:{\mathcal C}\rightarrow {\mathcal M}^{H}_{H}$ and $G=(\;\;)^{co H}: {\mathcal M}^{H}_{H}\rightarrow {\mathcal C}$ such that  $F\dashv G$, and they induce a pair of inverse equivalences. Thus, as it occurs in the Hopf algebra ambit, ${\mathcal M}^{H}_{H}$ is equivalent to the category of ${\Bbb F}$-vector spaces.

Hopf quasigroups admit a generalization to the weak seetting. The new notion, called weak Hopf quasigroup, was introduced in \cite{Asian} in a monoidal context and a  family of non trivial examples  can be obtained by working with bigroupoids, i.e., bicategories where every $1$-cell is an equivalence and every $2$-cell is an isomorphism (see Example 2.3 of \cite{Asian}). In \cite{MJM} we described these algebraic objects in terms of fusion morphisms and  in \cite{Asian}, for a weak Hopf quasigroup $H$ in a  braided monoidal category ${\mathcal C}$ with tensor product $\ot$, using the ideas proposed by Brzezi\'nski for Hopf quasigroups,  we introduce  the notion of Hopf module and  the category of Hopf modules ${\mathcal M}^{H}_{H}$. In this case, if we define $M^{co H}$ in the same way that in the weak Hopf algebra setting, we obtain the weak nonassociative version of the  Fundamental Theorem of Hopf modules in the following way: every Hopf module $M$ is isomorphic to  $M^{co H}\times H$ as Hopf modules, where $M^{co H}\times H$ is the image of the  same idempotent $\nabla_{M}$ used for Hopf modules associated to a weak Hopf algebra.  Moreover, in \cite{JPAA} we proved that $H_{L}$, the image of the target morphism, is a monoid, and then it is possible to take into consideration the category ${\mathcal C}_{H_{L}}$ to construct the tensor product $M^{co H}\ot_{H_{L}} H$, and, if the functor $-\ot H$ preserves coequalizers, to endow this object with a Hopf module structure. Unfortunately, unlike the case of weak Hopf algebras, it is not possible to assure in general that $M^{co H}\ot_{H_{L}} H$ is isomorphic to $M^{co H}\times H$. In order to find
sufficient conditions under which these objects are isomorphic in ${\mathcal M}^{H}_{H}$,  we introduce in \cite{Strong} the category of strong Hopf modules, denoted by   ${\mathcal SM}^{H}_{H}$ and  obtain that there exist two functors $F=-\ot_{H_{L}} H:{\mathcal C}_{H_{L}}\rightarrow {\mathcal SM}^{H}_{H}$ and  $G=(\;\;)^{co H}: {\mathcal SM}^{H}_{H}\rightarrow {\mathcal C}_{H_{L}}$ such that $F$ is left adjoint of $G$ and they induce a pair of inverse equivalences. In the Hopf quasigroup setting every Hopf module is strong, and then our results are the ones proved by Brzezi\'nski in \cite{Brz}. The same happens in the weak Hopf case  and then we generalize the theorem proved by B\"{o}hm, Nill and Szlach\'anyi in \cite{bohm}.  

Let  ${\mathcal C}$ be a braided monoidal category with tensor product $\ot$. Then for a weak Hopf quasigroup $H$ in  ${\mathcal C}$ and a right $H$-comodule magma $B$ (see \cite{JPAA} for the definition),  a question naturally arises: Is it possible to define a general category of $(H,B)$-Hopf modules and to prove a general theorem that permit to recover as particular instances the categorical equivalences cited in the previous paragraphs? The main contribution of this paper is to give a positive answer to this question. 

Now, we describe the paper in detail. After this introduction, for a weak Hopf quasigroup $H$ and a  right $H$-comodule magma $B$ in a strict braided monoidal category  $\mathcal C$ where every idempotent morphism splits,  in the second section we introduce the notion of anchor morphism $h:H\rightarrow B$ as an $H$-comodule morphism such that it is a morphism of unital magmas satisfying two suitable conditions. For an anchor morphism $h$,  in Definition \ref{H-D-mod}, we define the notion of strong $(H,B,h)$-Hopf module and prove some properties of these modules. We also find the condition under which  the subobject of coinvariants of $B$, defined as in the weak Hopf algebra context, i.e., $B^{co H}=\{b\in B \;| \; \rho_{B}(b)=b_{(0)}\ot \Pi_{H}^{L}(b_{(1)})\}$, is a monoid, and  construct the new category of strong $(H,B,h)$-Hopf modules, denoted by ${\mathcal S M}^{H}_{B}$. Moreover, if the category ${\mathcal C}$ admits coequalizers and the functors $-\ot B$ and $-\ot H$ preserve coequalizers, we prove in Theorem \ref{main0} that the Fundamental Theorem of Hopf Modules holds. In other words, for any strong $(H,B,h)$-Hopf module the objects $M$ and $M^{co H}\ot_{B^{co H}} B$ are isomorphic in ${\mathcal S M}^{H}_{B}$. This result admits as  particular instances the results with the same name cited in the previous paragraphs for associative and nonassociative (weak) Hopf structures. Finally, in the last section, we define the induction functor $F=-\ot_{B^{co H}} B:{\mathcal C}_{B^{co H}}\rightarrow {\mathcal  S M}^{H}_{B}$ and the functor of coinvariants $G=(\;\;)^{co H}: {\mathcal S M}^{H}_{B}\rightarrow {\mathcal C}_{B^{co H}}$, proving that  $F\dashv G$. Also, $F$ and $G$ is a pair of inverse equivalences and, therefore, ${\mathcal S M}^{H}_{B}$ is equivalent to the category of right $B^{co H}$-modules.

Throughout this paper $\mathcal C$ denotes a strict braided monoidal category with tensor product $\ot$, unit object $K$ and  braid $c$.  Without loss of generality, by the  coherence theorems, we can assume the  monoidal structure of ${\mathcal C}$ strict. Then, in this paper, we omit  explicitly  the associativity and unit constraints. For each object $M$ in  $ {\mathcal C}$, we denote the identity morphism by $id_{M}:M\rightarrow M$ and, for simplicity of notation, given objects $M$, $N$ and $P$ in ${\mathcal C}$ and a morphism $f:M\rightarrow N$, we write $P\ot f$ for $id_{P}\ot f$ and $f \ot P$ for $f\ot id_{P}$. We also assume that every idempotent morphism in  $ {\mathcal C}$ splits, i.e., if $\nabla:Y\rightarrow Y$ is such that $\nabla=\nabla\co\nabla$, there exist an object $Z$, called the image of $p$,  and morphisms $i:Z\rightarrow Y$ and $p:Y\rightarrow Z$ such that $\nabla=i\co p$ and $p\co i =id_{Z}$. The morphisms $p$ and $i$ will be called a factorization of $q$. Note that $Z$, $p$ and $i$ are unique up to isomorphism. The categories satisfying this property constitute a broad class that includes, among others, the categories with epi-monic decomposition for morphisms and categories with equalizers or coequalizers.  For example, complete bornological spaces is a symmetric monoidal closed category that is not  abelian, but it does have coequalizers (see \cite{Meyer}). On the other hand, let {\bf Hilb} be the category whose objects are complex Hilbert spaces and whose morphisms are the continuous linear maps. Then {\bf Hilb} is not an abelian and closed category but it is a symmetric monoidal category (see \cite{Kad}) with coequalizers.

As for prerequisites,  the reader is expected to be familiar with the notions of (co)unital (co)magma, (co)monoid, and morphism of (co)unital (co)magmas. By a unital  magma in ${\mathcal C}$ we understand a triple $A=(A, \eta_{A}, \mu_{A})$ where $A$ is an object in ${\mathcal C}$ and $\eta_{A}:K\rightarrow A$ (unit), $\mu_{A}:A\ot A \rightarrow A$ (product) are morphisms in ${\mathcal C}$ such that $\mu_{A}\co (A\ot \eta_{A})=id_{A}=\mu_{A}\co (\eta_{A}\ot A)$. If $\mu_{A}$ is associative, that is, $\mu_{A}\co (A\ot \mu_{A})=\mu_{A}\co (\mu_{A}\ot A)$, the unital magma will be called a monoid in ${\mathcal C}$.   Given two unital magmas
(monoids) $A= (A, \eta_{A}, \mu_{A})$ and $B=(B, \eta_{B}, \mu_{B})$, $f:A\rightarrow B$ is a morphism of unital magmas (monoids)  if $\mu_{B}\co (f\ot f)=f\co \mu_{A}$ and $ f\co \eta_{A}= \eta_{B}$. 

By duality, a counital comagma in ${\mathcal C}$ is a triple ${D} = (D, \varepsilon_{D}, \delta_{D})$ where $D$ is an object in ${\mathcal C}$ and $\varepsilon_{D}: D\rightarrow K$ (counit), $\delta_{D}:D\rightarrow D\ot D$ (coproduct) are morphisms in ${\mathcal C}$ such that $(\varepsilon_{D}\ot D)\co \delta_{D}= id_{D}=(D\ot \varepsilon_{D})\co \delta_{D}$. If $\delta_{D}$ is coassociative, that is, $(\delta_{D}\ot D)\co \delta_{D}= (D\ot \delta_{D})\co \delta_{D}$, the counital comagma will be called a comonoid. If ${D} = (D, \varepsilon_{D}, \delta_{D})$ and ${ E} = (E, \varepsilon_{E}, \delta_{E})$ are counital comagmas
(comonoids), $f:D\rightarrow E$ is  a morphism of counital comagmas (comonoids) if $(f\ot f)\co \delta_{D} =\delta_{E}\co f$ and  $\varepsilon_{E}\co f =\varepsilon_{D}.$

If  $A$, $B$ are unital magmas (monoids) in ${\mathcal C}$, the object $A\ot B$ is a unital  magma (monoid) in ${\mathcal C}$ where $\eta_{A\ot B}=\eta_{A}\ot \eta_{B}$ and $\mu_{A\ot B}=(\mu_{A}\ot \mu_{B})\co (A\ot c_{B,A}\ot B).$  In a dual way, if $D$, $E$ are counital comagmas (comonoids) in ${\mathcal C}$, $D\ot E$ is a  counital comagma (comonoid) in ${\mathcal C}$ where $\varepsilon_{D\ot E}=\varepsilon_{D}\ot \varepsilon_{E}$ and $\delta_{D\ot
E}=(D\ot c_{D,E}\ot E)\co( \delta_{D}\ot \delta_{E}).$

 Let  $A$ be a monoid. The pair
$(M,\phi_{M})$ is a right $A$-module if $M$ is an object in
${\mathcal C}$ and $\phi_{M}:M\otimes A\rightarrow M$ is a morphism
in ${\mathcal C}$ satisfying $\phi_{M}\circ(M\otimes
\eta_{A})=id_{M}$, $\phi_{M}\circ (\phi_{M}\otimes A)=\phi_{M}\circ
(M\otimes \mu_{A})$. Given two right ${A}$-modules $(M,\phi_{M})$
and $(N,\phi_{N})$, $f:M\rightarrow N$ is a morphism of right
${A}$-modules if $\phi_{N}\circ (f\otimes A)=f\circ \phi_{M}$.  If $D$ is a comonoid, the pair
$(M,\rho_{M})$ is a right $D$-comodule if $M$ is an object in
${\mathcal C}$ and $\rho_{M}:M\rightarrow M\ot D$ is a morphism
in ${\mathcal C}$ satisfying $(M\ot \varepsilon_{D})\co \rho_{M}=id_{M}$, $(\rho_{M}\ot D)\co \rho_{M}=(M\ot \delta_{D})\co \rho_{M}$. Given two right ${D}$-comodules $(M,\rho_{M})$
and $(N,\rho_{N})$, $f:M\rightarrow N$ is a morphism of right
${D}$-comodules if $(f\otimes D)\co \rho_{M}=\rho_{N}\co f$.

Finally, if $D$ is a comagma and $A$ a magma, given two morphisms $f,g:D\rightarrow A$ we will denote by $f\ast g$ its convolution product in ${\mathcal C}$, that is 
$$f\ast g=\mu_{A}\co (f\ot g)\co \delta_{D}.$$ 

\section{Doi-Hopf modules for weak Hopf quasigroups}

We begin this section by recalling the notion of weak Hopf quasigroup in a braided monoidal category  introduced in \cite{Asian}. In this reference the interested reader can find  an exhaustive list of  properties of weak Hopf quasigroups, that we will need along the paper.

\begin{definition}
\label{W-H-quasi}
{\rm A weak Hopf quasigroup $H$   in ${\mathcal C}$ is a unital magma $(H, \eta_H, \mu_H)$ and a comonoid $(H,\varepsilon_H, \delta_H)$ such that the following axioms hold:

\begin{itemize}

\item[(a1)] $\delta_{H}\co \mu_{H}=(\mu_{H}\ot \mu_{H})\co \delta_{H\ot H}.$

\item[(a2)] $\varepsilon_{H}\co \mu_{H}\co (\mu_{H}\ot H)=\varepsilon_{H}\co \mu_{H}\co (H\ot \mu_{H})$

\item[ ]$= ((\varepsilon_{H}\co \mu_{H})\ot (\varepsilon_{H}\co \mu_{H}))\co (H\ot \delta_{H}\ot H)$ 

\item[ ]$=((\varepsilon_{H}\co \mu_{H})\ot (\varepsilon_{H}\co \mu_{H}))\co (H\ot (c_{H,H}^{-1}\co\delta_{H})\ot H).$

\item[(a3)]$(\delta_{H}\ot H)\co \delta_{H}\co \eta_{H}=(H\ot \mu_{H}\ot H)\co ((\delta_{H}\co \eta_{H}) \ot (\delta_{H}\co \eta_{H}))$  \item[ ]$=(H\ot (\mu_{H}\co c_{H,H}^{-1})\ot H)\co ((\delta_{H}\co \eta_{H}) \ot (\delta_{H}\co \eta_{H})).$

\item[(a4)] There exists  $\lambda_{H}:H\rightarrow H$ in ${\mathcal C}$ (called the antipode of $H$) such that, if we denote the morphisms $id_{H}\ast \lambda_{H}$ by  $\Pi_{H}^{L}$ (target morphism) and $\lambda_{H}\ast id_{H}$ by $\Pi_{H}^{R}$ (source morphism),

\begin{itemize}

\item[(a4-1)] $\Pi_{H}^{L}=((\varepsilon_{H}\co \mu_{H})\ot H)\co (H\ot c_{H,H})\co ((\delta_{H}\co \eta_{H})\ot
H).$

\item[(a4-2)] $\Pi_{H}^{R}=(H\ot(\varepsilon_{H}\co \mu_{H}))\co (c_{H,H}\ot H)\co (H\ot (\delta_{H}\co \eta_{H})).$

\item[(a4-3)]$\lambda_{H}\ast \Pi_{H}^{L}=\Pi_{H}^{R}\ast \lambda_{H}= \lambda_{H}.$

\item[(a4-4)] $\mu_H\co (\lambda_H\ot \mu_H)\co (\delta_H\ot H)=\mu_{H}\co (\Pi_{H}^{R}\ot H).$

\item[(a4-5)] $\mu_H\co (H\ot \mu_H)\co (H\ot \lambda_H\ot H)\co (\delta_H\ot H)=\mu_{H}\co (\Pi_{H}^{L}\ot H).$

\item[(a4-6)] $\mu_H\co(\mu_H\ot \lambda_H)\co (H\ot \delta_H)=\mu_{H}\co (H\ot \Pi_{H}^{L}).$

\item[(a4-7)] $\mu_H\co (\mu_H\ot H)\co (H\ot \lambda_H\ot H)\co (H\ot \delta_H)=\mu_{H}\co (H\ot \Pi_{H}^{R}).$

\end{itemize}

\end{itemize}

}

\end{definition}

Note that, if in the previous definition the triple $(H, \eta_H, \mu_H)$ is a monoid, we obtain the notion of weak Hopf algebra in a symmetric monoidal category. Then, if ${\mathcal C}$ is the category of vector spaces over a field ${\Bbb F}$, we have  the original definition of weak Hopf algebra introduced by B\"{o}hm, Nill and Szlach\'anyi in \cite{bohm}. On the other hand, under these conditions, if  $\varepsilon_H$ and $\delta_H$ are  morphisms of unital magmas (equivalently, $\eta_{H}$, $\mu_{H}$ are morphisms of counital comagmas), $\Pi_{H}^{L}=\Pi_{H}^{R}=\eta_{H}\ot \varepsilon_{H}$. As a consequence, conditions (a2), (a3), (a4-1)-(a4-3) trivialize, and we get the notion of Hopf quasigroup defined  by Klim and Majid in \cite{Majidesfera}.  More concretely, a Hopf quasigroup $H$ in
${\mathcal C}$  is a unital magma $(H,\eta_H,\mu_H)$ and a comonoid $(H,\varepsilon_H,\delta_H)$ satisfying that  $\varepsilon_H$ and $\delta_H$ are morphisms of unital magmas (equivalently, $\eta_H$ and $\mu_H$ are morphisms of counital comagmas), and such that there exists a morphism $\lambda_{H}:H\rightarrow H$
in ${\mathcal C}$, called the antipode of $H$, for which  
\begin{equation}
\label{leftHqg}
\mu_H\circ (\lambda_H\ot \mu_H)\circ (\delta_H\ot H)=
\varepsilon_H\ot H=
\mu_H\circ (H\ot \mu_H)\circ (H\ot \lambda_H\ot H)\circ (\delta_H\ot H)
\end{equation}
 and 
 \begin{equation}
\label{rightHqg}
\mu_H\circ (\mu_H\ot H)\circ (H\ot \lambda_H\ot H)\circ (H\ot \delta_H)=
H\ot \varepsilon_H=
\mu_H\circ(\mu_H\ot \lambda_H)\circ (H\ot \delta_H) 
\end{equation}
hold. Then, as a consequence,  we have (a1)  and the following identities
\begin{equation}
\label{e-m-d-eps}
\varepsilon_{H}\co \eta_{H}=id_{K},\;\;
\varepsilon_{H}\co \mu_{H}=\varepsilon_{H}\ot \varepsilon_{H},\;\;
\delta_{H}\co \eta_{H}=\eta_{H}\ot \eta_{H}.
\end{equation}

By Proposition 3.2 of  \cite{Asian} we know that the antipode  of a weak Hopf quasigroup  is unique, and satisfies that $\lambda_{H}\co \eta_{H}=\eta_{H}$,  $\varepsilon_{H}\co\lambda_{H}=\varepsilon_{H}$. Also, by Theorem 3.19 of \cite{Asian}, we have that it is antimultiplicative and anticomultiplicative. Moreover, if we define the morphisms $\overline{\Pi}_{H}^{L}$ and $\overline{\Pi}_{H}^{R}$ by 
$$\overline{\Pi}_{H}^{L}=(H\ot (\varepsilon_{H}\co \mu_{H}))\co ((\delta_{H}\co \eta_{H})\ot H),\;\;\;\overline{\Pi}_{H}^{R}=((\varepsilon_{H}\co \mu_{H})\ot H)\co (H\ot (\delta_{H}\co \eta_{H})),$$
we proved in Proposition 3.4 of \cite{Asian}, that $\Pi_{H}^{L}$, $\Pi_{H}^{R}$, $\overline{\Pi}_{H}^{L}$ and 
$\overline{\Pi}_{H}^{R}$ are idempotent. 

\begin{lemma}
\label{monoid-hl}
Let $H$ be a weak Hopf quasigroup and $\Pi\in\{\Pi_{H}^{L},\Pi_{H}^{R},\overline{\Pi}_{H}^{L}, \overline{\Pi}_{H}^{R}\}$ . The following identities hold:
\begin{equation}
\label{monoid-hl-1}
\mu_{H}\co ((\mu_{H}\co (\Pi\ot H))\ot H)=\mu_{H}\co (\Pi\ot \mu_{H}), 
\end{equation}
\begin{equation}
\label{monoid-hl-2}
\mu_{H}\co (H\ot (\mu_{H}\co (\Pi\ot H)))=\mu_{H}\co ((\mu_{H}\co (H\ot \Pi))\ot H), 
\end{equation}
\begin{equation}
\label{monoid-hl-3}
\mu_{H}\co (H\ot (\mu_{H}\co (H\ot \Pi)))=\mu_{H}\co (\mu_{H}\ot \Pi). 
\end{equation}
\end{lemma}

\begin{proof}

The proof for $\Pi_{H}^{L}$ is  in Proposition 2.4 of \cite{JPAA} and in a similar way we can prove the result for $\Pi_{H}^{R}$  (see also Proposition 2.3 of \cite{MJM}). The equalities for $\overline{\Pi}_{H}^{L}$ and $\overline{\Pi}_{H}^{R}$ follow from  Proposition 3.11 of \cite{Asian}. 

\end{proof}

If $H_{L}$ is the image of the idempotent morphism $\Pi_{H}^{L}$, and 
$p_{L}:H\rightarrow H_{L}$, $i_{L}:H_{L}\rightarrow H$ are the
morphisms such that $\Pi_{H}^{L}=i_{L}\co p_{L}$ and $p_{L}\co
i_{L}=id_{H_{L}}$, by Proposition 3.13 of \cite{Asian}, $i_{L}$ is the equalizer of $\delta_{H}$ and $(H\ot \Pi_{H}^{L}) \co
\delta_{H}$ and  $p_{L}$ is the coequalizer of $\mu_{H}$ and $\mu_{H}\co (H\ot \Pi_{H}^{L})$. Then the triple $(H_{L}, \varepsilon_{H_{L}}=\varepsilon_{H}\co i_{L}, \delta_{H}=(p_{L}\ot
p_{L})\co \delta_{H}\co i_{L})$ is a comonoid in ${\mathcal C}$, and as a consequence of Lemma \ref{monoid-hl},  $(H_{L},
\eta_{H_{L}}=p_{L}\co \eta_{H}, \mu_{H_{L}}=p_{L}\co \mu_{H}\co
(i_{L}\ot i_{L}))$ is a monoid in ${\mathcal C}$. Following Remark 3.15 of \cite{Asian}, we have similar results for the image of the idempotent morphism $\Pi_{H}^{R}$ denoted by $H_{R}$. 

\begin{definition}
\label{H-comodmag}
{\rm  Let $H$ be a weak Hopf quasigroup
and let $B$ be a unital magma, which is also a right
$H$-comodule with coaction $\rho_{B}:B\rightarrow B\ot H$  such that 
\begin{equation}
\label{chmagma}
\mu_{B\ot H}\co (\rho_{B}\ot
\rho_{B})=\rho_{B}\co \mu_{B}.
\end{equation}
 We will say that $(B,\rho_{B})$ is a right $H$-comodule
magma if any of the following equivalent conditions hold:
\begin{itemize}
\item[(b1)]$(\rho_{B}\ot H)\co \rho_{B}\co
\eta_{B}=(B\ot (\mu_{H}\co c_{H,H}^{-1})\ot H)\co
((\rho_{B}\co \eta_{B})\ot (\delta_{H}\co \eta_{H})). $
\item[(b2)]$(\rho_{B}\ot H)\co \rho_{B}\co
\eta_{B}=(B\ot \mu_{H}\ot H)\co ((\rho_{B}\co \eta_{B})\ot (\delta_{H}\co \eta_{H})). $
\item[(b3)]$(B\ot \overline{\Pi}_{H}^{R})\co
\rho_{B}=(\mu_{B}\ot H)\co (B\ot (\rho_{B}\co \eta_{B})),$
\item[(b4)]$(B\ot \Pi_{H}^{L})\co \rho_{B}= ((\mu_{B}\co
c_{B,B}^{-1})\ot H)\co (B\ot (\rho_{B}\co \eta_{B})).$
\item[(b5)]$(B\ot \overline{\Pi}_{H}^{R})\co \rho_{B}\co
\eta_{B}=\rho_{B}\co\eta_{B}.$
\item[(b6)]$ (B\ot \Pi_{H}^{L})\co \rho_{B}\co
\eta_{B}=\rho_{B}\co\eta_{B}.$
\end{itemize}

This definition is similar to the notion of right $H$-comodule monoid  in the weak Hopf algebra setting and the proof for the equivalence of (b1)-(b6) also follows in a similar way.  

Note that, if $H$ is a Hopf quasigroup  and $B$ is a unital magma which is also a right
$H$-comodule with coaction $\rho_{B}:B\rightarrow B\ot H$, we will say that $(B,\rho_{B})$ is a right $H$-comodule magma if it satisfies (\ref{chmagma}) and $\rho_{B}\co \eta_{B}=\eta_{H}\ot \eta_{B}$. In this case (b1)-(b6) trivialize.
}
\end{definition}

\begin{example}
\label{ex-hcm}
{\rm 
1. If $H$ is a (weak) Hopf quasigroup, $(H, \delta_H)$ is a right $H$-comodule magma. 

2. Let $H$ be a cocommutative weak Hopf quasigrop and assume that ${\mathcal C}$ is symmetric. Then, $c_{H,H}\co \delta_{H}=\delta_{H}$ and, by Theorem 3.22 of \cite{Asian}, $\lambda_{H}^{2}=id_{H}$. If  we denote by  $H^{op}$ the unital magma $H^{op}=(H, \eta_{H^{op}}=\eta_{H}, \mu_{H^{op}}=\mu_{H}\co c_{H,H})$, we have that $(H^{op}, \rho_{H^{op}}=(H\ot \lambda_{H})\co \delta_{H})$ is an example of right $H$-comodule magma.  Indeed, first note that $H^{op}$ is a unital magma. Also $(H\ot \varepsilon_{H})\co\rho_{H^{op}}=id_{H}$ because $\lambda_{H}$ presevers the counit. On the other hand,   if $H$ is cocommutative, the equality
\begin{equation}
\label{2ex}
\delta_{H}\co \lambda_{H}=(\lambda_{H}\ot \lambda_{H})\co \delta_{H}
\end{equation}
holds. Then, by the coassociativity of $\delta_{H}$ and (\ref{2ex}), we obtain that $(\rho_{H^{op}}\ot  H)\co \rho_{H^{op}}=(H\ot \delta_{H})\co \rho_{H^{op}}$, an we have that $(H^{op}, \rho_{H^{op}})$ is a right $H$-comodule. Finally, 

\begin{itemize}
\item[ ]$\hspace{0.38cm} \rho_{H^{op}}\co \mu_{H^{op}}$

\item[ ]$= (\mu_{H}\ot (\lambda_{H}\co \mu_{H}))\co (c_{H,H}\ot c_{H,H})\co \delta_{H\ot H} $ {\scriptsize   ({\blue naturality of $c$ and $c^{2}=id_{H}$})}

\item[ ]$=(\mu_{H^{op}}\ot (\mu_{H}\co (\lambda_{H}\ot \lambda_{H})))\co  \delta_{H\ot H} $ {\scriptsize  ({\blue (53) of \cite{Asian} and $c^{2}=id_{H}$})}

\item[ ]$=(\mu_{H^{op}}\ot \mu_{H})\co  (H\ot c_{H,H}\ot H)\co (\rho_{H^{op}}\ot \rho_{H^{op}})$ {\scriptsize  ({\blue naturality of $c$})}

\end{itemize}

and 

\begin{itemize}
\item[ ]$\hspace{0.38cm} \rho_{H^{op}}\co \eta_{H^{op}}$

\item[ ]$= (H\ot (\lambda_{H}\co \Pi_{H}^{L})) \co \delta_{H}\co \eta_{H} $ {\scriptsize  ({\blue (21) of \cite{Asian}})}

\item[ ]$=(H\ot (\lambda_{H}\co \overline{\Pi}_{H}^{L})) \co \delta_{H}\co \eta_{H}  $ {\scriptsize  ({\blue if $H$ is cocommutative $\Pi_{H}^{L}=\overline{\Pi}_{H}^{L}$})}

\item[ ]$= (H\ot  \Pi_{H}^{L}) \co \delta_{H}\co \eta_{H} $ {\scriptsize ({\blue (39) of \cite{Asian}})}

\item[ ]$=(H\ot  (\Pi_{H}^{L}\co \Pi_{H}^{L})) \co \delta_{H}\co \eta_{H} $ {\scriptsize  ({\blue $\Pi_{H}^{L}$ is idempotent})}

\item[ ]$= (H\ot \Pi_{H}^{L}) \co \rho_{H^{op}}\co \eta_{H} $ {\scriptsize  ({\blue  (21) of \cite{Asian}})}.

\end{itemize}

Therefore, $(H^{op}, \rho_{H^{op}})$ is a right $H$-comodule magma. Note that, if $H$ is a Hopf quasigroup we have the same example.

3. By Example 3.1 of \cite{our2} we have the following. Let $H$ be a Hopf quasigroup and $A$ a  unital magma in ${\mathcal C}$. If there exists a morphism $\varphi_A:H\ot A\to A$ such that
\begin{equation}
\label{et1}
\varphi_A\co(\eta_H\ot A)= id_A,
\end{equation}
\begin{equation}
\label{et2}
\varphi_A\co(H\ot\eta_A)=\varepsilon_H\ot\eta_A,
\end{equation}
hold, then the smash product $A\sharp H=(A\otimes H,\eta_{A\sharp H},\mu_{A\sharp H})$ defined by 
$$\eta_{A\sharp H}=\eta_A\ot\eta_H,\;\;
\mu_{A\sharp H}=(\mu_A\ot \mu_H)\co(A\ot \psi_{H}^{A}\ot H),$$
where 
$$\psi_{H}^{A}=(\varphi_{A}\ot H)\co (H\ot c_{H,A})\co (\delta_{H}\ot A),$$ is a  right $H$-comodule magma with comodule structure given by 
$$\varrho_{A\sharp H}=A\ot\delta_H.$$

4. Let $H$, $B$ two Hopf quasigroups. Assume that there exists a morphism of Hopf quasigroups  $g:B\rightarrow H$, i.e., a morphism of unital magmas and comonoids. Then, $(B,\rho_{B}=(B\ot g)\co \delta_{B})$ is an example of right $H$-comodule magma.

}
\end{example}

\begin{definition}
Let $H$ be a weak Hopf quasigroup and let $(B,\rho_{B})$ be a right $H$-comodule
magma. We will say that $h:H\rightarrow B$ is an integral if it is a morphism of right $H$-comodules. The integral will be called total if $h\co \eta_{H}=\eta_{B}$.
\end{definition}

\begin{proposition}
\label{idem-B}
Let $H$ be a weak Hopf quasigroup and let $(B,\rho_{B})$ be a right $H$-comodule
magma. Let $h:H\rightarrow B$ be a total integral. The endomorphism $q_{B}:=\mu_{B}\co (B\ot (h\co \lambda_{H}))\co \rho_{B}:B\rightarrow B$  satisfies 
\begin{equation}
\label{idemp-0}
\rho_{B}\co q_{B}=(B\ot \Pi_{H}^{L})\co\rho_{B}\co q_{B},
\end{equation}
\begin{equation}
\label{idemp-1}
\rho_{B}\co q_{B}=(B\ot \overline{\Pi}_{H}^{R})\co\rho_{B}\co q_{B},
\end{equation}
and, as a consequence, $q_{B}$ is an idempotent morphism. Moreover, if $B^{co H}$ (object of coinvariants)  is the image of $q_{B}$ and $p_{B}:B\rightarrow B^{co H}$, $i_{B}:B^{co H}\rightarrow B$ are the morphisms such that $q_{B}=i_{B}\co p_{B}$ and 
 $id_{B^{co H}}=p_{B}\co i_{B}$, 
$$
\setlength{\unitlength}{3mm}
\begin{picture}(30,4)
\put(3,2){\vector(1,0){4}} \put(11,2.5){\vector(1,0){10}}
\put(11,1.5){\vector(1,0){10}} \put(1,2){\makebox(0,0){$B^{co H}$}}
\put(9,2){\makebox(0,0){$B$}} \put(24,2){\makebox(0,0){$B\ot H,$}}
\put(5.5,3){\makebox(0,0){$i_{B}$}}
\put(16,3.5){\makebox(0,0){$ \rho_{B}$}}
\put(16,0.15){\makebox(0,0){$(B\ot \overline{\Pi}_{H}^{R})\co\rho_{B}$}}
\end{picture}
$$ 

$$
\setlength{\unitlength}{3mm}
\begin{picture}(30,4)
\put(3,2){\vector(1,0){4}} \put(11,2.5){\vector(1,0){10}}
\put(11,1.5){\vector(1,0){10}} \put(1,2){\makebox(0,0){$B^{co H}$}}
\put(9,2){\makebox(0,0){$B$}} \put(24,2){\makebox(0,0){$B\ot H,$}}
\put(5.5,3){\makebox(0,0){$i_{B}$}}
\put(16,3.25){\makebox(0,0){$ \rho_{B}$}}
\put(16,0.15){\makebox(0,0){$(B\ot \Pi_{H}^{L})\co\rho_{B}$}}
\end{picture}
$$
are equalizer diagrams. 
\end{proposition}

\begin{proof} 
First note that, by the naturality of $c$, the condition of right $H$-comodule morphism for $h$, and  the equality $h\co \eta_{H}=\eta_{B}$, we obtain that 
\begin{equation}
\label{idemp-2}
h\co \Pi_{H}^{R}=(B\ot (\varepsilon_{H}\co \mu_{H}))\co (c_{H,B}\ot H)\co (H\ot (\rho_{B}\co \eta_{B}))
\end{equation}
holds.  Then, as a consequence, by (\ref{chmagma}) and the properties of $\varepsilon_{H}$ and $\eta_{B}$, we have 
\begin{equation}
\label{idemp-3}
\mu_{B}\co (B\ot (h\co \Pi_{H}^{R}))\co \rho_{B}=id_{B}.
\end{equation}

Also, 
\begin{itemize}
\item[ ]$\hspace{0.38cm}\rho_{B}\co q_{B}$

\item[ ]$= \mu_{B\ot H}\co (\rho_{B}\ot (\rho_{B}\co h\co \lambda_{H}))\co \rho_{B}$
{\scriptsize ({\blue (\ref{chmagma})})}

\item[ ]$=(\mu_{B}\ot H)\co (B\ot ((h\ot \mu_{H})\co (c_{H,H}\ot H)\co (H\ot (\delta_{H}\co \lambda_{H}))\co \delta_{H}))\co \rho_{B}$ {\scriptsize ({\blue comodule condition for $B$ and }} 
\item[ ]$\hspace{0.38cm} $ {\scriptsize {\blue comodule morphism condition for $h$})}

\item[ ]$=(\mu_{B}\ot H)\co (B\ot ((h\ot \mu_{H})\co (c_{H,H}\ot H)\co (H\ot c_{H,H})\co (((H\ot \lambda_{H})\co \delta_{H})\ot \lambda_{H})\co \delta_{H}))\co \rho_{B}$ 
\item[ ]$\hspace{0.38cm} ${\scriptsize ({\blue  anticomultiplicativity of $\lambda_{H}$ and coassociativity of $\delta_{H}$})}

\item[ ]$=((\mu_{B}\co (B\ot h))\ot \Pi_{H}^{L})\co (B\ot (c_{H,H}\co (H\ot \lambda_{H})\co \delta_{H}))\co \rho_{B}$ {\scriptsize ({\blue naturality of $c$})},
\end{itemize}
 and then using that $\Pi_{H}^{L}$ is an idempotent morphism we obtain (\ref{idemp-0}). Therefore, by (34) of \cite{Asian}, we have (\ref{idemp-1}) because 
$$(B\ot \overline{\Pi}_{H}^{R})\co\rho_{B}\co q_{B}=(B\ot (\overline{\Pi}_{H}^{R}\co \Pi_{H}^{L}))\co\rho_{B}\co q_{B}=(B\ot \Pi_{H}^{L})\co\rho_{B}\co q_{B}=\rho_{B}\co q_{B}.$$

Then, $q_{B}$ is an idempotent morphism. Indeed, 
\begin{itemize}
\item[ ]$\hspace{0.38cm} q_{B}\co q_{B}$

\item[ ]$= \mu_{H} \co (B\ot (h\co \lambda_{H}\co \overline{\Pi}_{H}^{R}) )\co \rho_{B}\co q_{B}${\scriptsize ({\blue (\ref{idemp-1})})}

\item[ ]$=\mu_{B} \co (B\ot  (h\co \Pi_{H}^{R}))\co \rho_{B}\co q_{B}$ {\scriptsize  ({\blue (40) of \cite{Asian}})}

\item[ ]$=q_{B}$ {\scriptsize  ({\blue (\ref{idemp-3})})}.

\end{itemize}

On the other hand, by (40)  of \cite{Asian} and (\ref{idemp-3}), 
$$\mu_{B} \co (B\ot (h\co \lambda_{H}\co \overline{\Pi}_{H}^{R}) )\co \rho_{B}=\mu_{B} \co (B\ot (h\co \Pi_{H}^{R})) )\co \rho_{B}=id_{B}.$$

Then, 
$$
\setlength{\unitlength}{3mm}
\begin{picture}(30,4)
\put(3,2){\vector(1,0){4}} \put(11,2.5){\vector(1,0){10}}
\put(11,1.5){\vector(1,0){10}} \put(1,2){\makebox(0,0){$B^{co H}$}}
\put(9,2){\makebox(0,0){$B$}} \put(24,2){\makebox(0,0){$B\ot H$}}
\put(5.5,3){\makebox(0,0){$i_{B}$}}
\put(16,3.5){\makebox(0,0){$ \rho_{B}$}}
\put(16,0.15){\makebox(0,0){$(B\ot \overline{\Pi}_{H}^{R})\co\rho_{B}$}}
\end{picture}
$$
is a split cofork  (see \cite{Mac}) and thus an equalizer diagram. As a consequence, by (33) of \cite{Asian}, we have 
$$(B\ot \Pi_{H}^{L})\co\rho_{B}\co q_{B}=(B\ot (\Pi_{H}^{L}\co \overline{\Pi}_{H}^{R}))\co\rho_{B}\co q_{B}=(B\ot \overline{\Pi}_{H}^{R})\co\rho_{B}\co q_{B}=\rho_{B}\co q_{B}$$
and  $q_{B}$ equalizes $\rho_{B}$ and $(B\ot \Pi_{H}^{L})\co\rho_{B}$. Therefore,  $i_{B}$ equalizes $\rho_{B}$ and $(B\ot \Pi_{H}^{L})\co\rho_{B}$ because $p_{B}\co i_{B}=id_{B^{coH}}$. Moreover, if $t:D\rightarrow B$ is such that $(B\ot \Pi_{H}^{L})\co\rho_{B}\co t=\rho_{B}\co t$, composing with $B\ot \overline{\Pi}_{H}^{R}$, we obtain that $(B\ot \overline{\Pi}_{H}^{R})\co\rho_{B}\co t=\rho_{B}\co t$. Thus, there exists a unique morphism $t^{\prime}:D\rightarrow B^{coH}$ such that $i_{B}\co t^{\prime}=t$. Then,    
$$
\setlength{\unitlength}{3mm}
\begin{picture}(30,4)
\put(3,2){\vector(1,0){4}} \put(11,2.5){\vector(1,0){10}}
\put(11,1.5){\vector(1,0){10}} \put(1,2){\makebox(0,0){$B^{co H}$}}
\put(9,2){\makebox(0,0){$B$}} \put(24,2){\makebox(0,0){$B\ot H$}}
\put(5.5,3){\makebox(0,0){$i_{B}$}}
\put(16,3.5){\makebox(0,0){$ \rho_{B}$}}
\put(16,0.15){\makebox(0,0){$(B\ot \Pi_{H}^{L})\co\rho_{B}$}}
\end{picture}
$$
is an equalizer diagram.

\end{proof}

\begin{remark}
Note that, under the conditions of the previous proposition, the object of coinvariants is independent of the total integral $h$ because it is the equalizer object of $\rho_{B}$ and $(B\ot \overline{\Pi}_{H}^{R})\co\rho_{B}$ (or equivalently, $\rho_{B}$ and $(B\ot \Pi_{H}^{L})\co\rho_{B}$).

Moreover, (see \cite{JPAA}) the triple $(B^{co H}, \eta_{B^{co H}}, \mu_{B^{co H}})$ is a unital magma (the  submagma of coinvariants of $B$), where $\eta_{B^{co H}}:K\rightarrow B^{co H}$, $\mu_{B^{co H}}:B^{co H}\ot B^{co H}\rightarrow B^{co H}$ are the factorizations through $i_{B}$ of the morphisms $\eta_B$ and $\mu_B\co (i_B\ot i_B)$, respectively. Therefore, $\eta_{B^{co H}}$ is the unique morphism such that 
\begin{equation}
\label{etaH}
\eta_{B}=i_{B}\co \eta_{B^{co H}},
\end{equation}
and $\mu_{B^{co H}}$ is the unique morphism satisfying 
\begin{equation}
\label{muH}
\mu_{B}\co (i_{B}\ot i_{B})=i_{B}\co \mu_{B^{co H}}.
\end{equation}

Thus, 
\begin{equation}
\label{etaH1}
\eta_{B^{co H}}=p_{B}\co \eta_{B},
\end{equation}
and 
\begin{equation}
\label{muH1}
\mu_{B^{co H}}=p_{B}\co \mu_{B}\co (i_{B}\ot i_{B}).
\end{equation}

In what follows, the object of coinvariants $B^{co H}$ will be called the submagma of coinvariants of $B$. 
Note that, if $B=H$ and $\rho_{B}=\delta_{H}$, the submagma of coinvariants is $H^{coH}=H_{L}$ and then, in this case, it is a monoid.

\end{remark}

\begin{definition}
\label{anchor}
Let $H$ be a weak Hopf quasigroup and let $(B,\rho_{B})$ be a right $H$-comodule
magma. We will say that $h:H\rightarrow B$ is an anchor morphism if it is a multiplicative total integral (i.e., a right $H$-comodule morphism such that it is a morphism of unital magmas) and the following equalities hold:
\begin{itemize}
\item[(c1)]$\mu_{B}\co ((\mu_{B}\co (B\ot h))\ot (h\co \lambda_{H}))\co (B\ot \delta_{H})=\mu_{B}\co (B\ot (h\co \Pi_{H}^{L})).$
\item[(c2)]$\mu_{B}\co ((\mu_{B}\co (B\ot (h\co \lambda_{H})))\ot h)\co (B\ot \delta_{H})=\mu_{B}\co (B\ot (h\co \Pi_{H}^{R})). $
\end{itemize}

Note that, if the product on $B$ is associative, every multiplicative total integral $h$ satisfies (c1)-(c2) and therefore is an anchor morphism. Also, using that $h$ is a comodule morphism, the condition (c1) can be rewritten as 
\begin{equation}
\label{re-c1}
\mu_{B}\co (\mu_{B}\ot (h\co \lambda_{H}))\co (B\ot (\rho_{B}\co h))=\mu_{B}\co (B\ot (h\co \Pi_{H}^{L})).
\end{equation}

Finally, if $H$ is a Hopf quasigroup, (c1) and (c2) are 
$$\mu_{B}\co ((\mu_{B}\co (B\ot h))\ot (h\co \lambda_{H}))\co (B\ot \delta_{H})=B\ot \varepsilon_{H}=\mu_{B}\co ((\mu_{B}\co (B\ot (h\co \lambda_{H})))\ot h)\co (B\ot \delta_{H}), $$
or, in an equivalent way, $B$ is a cleft right $H$-comodule algebra in the sense of Definition 3.1 of \cite{our2}.

\begin{example}
\label{anch-ex}
{\rm 
1. By the definition of weak Hopf quasigroup and (40) of \cite{Asian}, the identity morphism $id_{H}$  is an anchor morphism for the right $H$-comodule magma $(H,\delta_{H})$. Also, if $H$ is a Hopf quasigrup, the identity of $H$ is an anchor morphism.

2. By the second point of Example \ref{ex-hcm} we know that, if $H$ is a cocommutative weak Hopf quasigroup and ${\mathcal C}$ is symmetric, $(H^{op}, \rho_{H^{op}}=(H\ot \lambda_{H})\co \delta_{H})$ is an example of right $H$-comodule magma. Note that in this case we have that $\lambda_{H}\co \eta_{H}=\eta_{H}$ and 

\begin{itemize}
\item[ ]$\hspace{0.38cm} \rho_{H^{op}}\co \lambda_{H}$

\item[ ]$= (\lambda_{H}\ot \lambda_{H}^2)\co \delta_{H} ${\scriptsize ({\blue (\ref{2ex})})}

\item[ ]$=(\lambda_{H}\ot H)\co \delta_{H} $ {\scriptsize  ({\blue Theorem 3.22 of \cite{Asian}, i.e., $\lambda_{H}^2=id_{H}$})}

\end{itemize} 

Therefore $\lambda_{H}$ is a total integral. Moreover, by (52) of \cite{Asian}, we obtain that 
$\mu_{H^{op}}\co (\lambda_{H}\ot \lambda_{H})=\lambda_{H}\co \mu_{H}$. Finally, 

\begin{itemize}
\item[ ]$\hspace{0.38cm} \mu_{H^{op}}\co ((\mu_{H^{op}}\co (H\ot \lambda_{H}))\ot (\lambda_{H}\co \lambda_{H}))\co (H\ot \delta_{H})$

\item[ ]$=\mu_{H}\co (H\ot \mu_{H})\co (c_{H,H}\ot H)\co (H\ot c_{H,H})\co  (c_{H,H}\ot H)\co (H\ot c_{H,H}) \co (H\ot ((H\ot \lambda_{H})\co \delta_{H}))$
\item[ ]$\hspace{0.38cm}${\scriptsize ({\blue Theorem 3.22 of \cite{Asian}, naturality of $c$ and coassociativity of $\delta_{H}$})}

\item[ ]$=\mu_H\co (H\ot \mu_H)\co (H\ot \lambda_H\ot H)\co (\delta_H\ot H)\co c_{H,H} $ {\scriptsize  ({\blue naturality of $c$ and  $c^2=id_{H}$})}

\item[ ]$=\mu_{H}\co (\Pi_{H}^{L}\ot H)\co c_{H,H} $ {\scriptsize  ({\blue (a4-5) of Definition \ref{W-H-quasi}})}

\item[ ]$= \mu_{H^{op}}\co (H\ot \Pi_{H}^{L}) ${\scriptsize ({\blue naturality of $c$})}

\item[ ]$=\mu_{H^{op}}\co (H\ot (\lambda_{H}\co \overline{\Pi}_{H}^{L}))$ {\scriptsize  ({\blue (39) of \cite{Asian}})}

\item[ ]$=\mu_{H^{op}}\co (H\ot (\lambda_{H}\co \Pi_{H}^{L}))$  {\scriptsize  ({\blue $H$ is cocommutative, $\Pi_{H}^{L}=\overline{\Pi}_{H}^{L}$})}, 

\end{itemize} 
and, by similar arguments, 
$$\mu_{H^{op}}\co ((\mu_{H^{op}}\co (H\ot (\lambda_{H}\co \lambda_{H})))\ot \lambda_{H})\co (H\ot \delta_{H})=\mu_{H^{op}}\co (H\ot (\lambda_{H}\co \Pi_{H}^{R})). $$

Therefore, $\lambda_{H}$ is an anchor morphism for $(H^{op}, \rho_{H^{op}})$.  Of course, the same result holds for cocommutative Hopf quasigroups.

3. In the third point of Example \ref{ex-hcm} we saw that if $H$ is a Hopf quasigroup, $A$ is a unital magma in ${\mathcal C}$, and there  exists a morphism $\varphi_A:H\ot A\to A$ satisfying (\ref{et1}), (\ref{et2}), the smash product $A\sharp H$ is a right $H$-comodule magma with coaction $\rho_{A\sharp H}=A\ot \delta_{H}$. By Example 3.1 of \cite{our2}, we have that $h=\eta_{A}\ot H:H\rightarrow A\sharp H$  is a total integral. On the other hand, 

\begin{itemize}
\item[ ]$\hspace{0.38cm} \mu_{A\sharp H}\co (h\ot h)$

\item[ ]$= ((\varphi_{A}\co (H\ot \eta_{A}))\ot \mu_{H})\co (\delta_{H}\ot H)${\scriptsize ({\blue unit properties and naturality of $c$})}

\item[ ]$= h\co \mu_{H}$ {\scriptsize  ({\blue (\ref{et2}) and counit properties})}

\end{itemize}
and then $h$ is multiplicative. Also, by similar arguments, we have 
\begin{equation}
\label{nwx}
\mu_{A\sharp H}\co (A\ot H\ot h)=A\ot \mu_{H}.
\end{equation}

Thus, by (\ref{nwx}) and (\ref{rightHqg}), the identities  
$$\mu_{A\sharp H}\co ((\mu_{A\sharp H}\co (A\ot H\ot h))\ot (h\co \lambda_{H}))\co (A\ot H\ot \delta_{H})=A\ot 
(\mu_H\circ(\mu_H\ot \lambda_H)\circ (H\ot \delta_H))=A\ot H\ot \varepsilon_{H},$$
hold. Similarly, by (\ref{nwx}) and (\ref{leftHqg}) we obtain that 
$$\mu_{A\sharp H}\co ((\mu_{A\sharp H}\co (A\ot H\ot (h\co \lambda_{H})))\ot h)\co (A\ot H\ot \delta_{H})=A\ot H\ot \varepsilon_{H}.$$

Therefore, $h=\eta_{A}\ot H$ is an anchor morphism.

4.  Assume that $H$ and $B$ are Hopf quasigrous in ${\mathcal C}$. Let $g:B\to H$, $f:H\to B$ be  morphisms of Hopf quasigroups such that $g\co f=id_H$. Consider the right $H$-comodule structure on $B$ defined in fourth point of Example \ref{ex-hcm}. Then, $f$ is an anchor morphism. Indeed, first note that $f\co \eta_{H}=\eta_{B}$, $f\co \mu_{H}=\mu_{B}\co (f\ot f)$ hold because $f$ is a morphism of unital magmas. Also, $f$ is a comodule morphism, i.e.,  
$\rho_{B}\co f=(f\ot H)\co \delta_{H}$, because $f$ is a comonoid morphism and $g\co f=id_H$. On the other hand, by Lemma 1.4 of \cite{our1}, we have that 
\begin{equation}
\label{elambda}
\lambda_{B}\co f=f\co \lambda_{H}
\end{equation}
holds, and then, using that $f$ is a comonoid morphism, we get 
$$\mu_{B}\co ((\mu_{B}\co (B\ot f))\ot (f\co \lambda_{H}))\co (B\ot \delta_{H})=\mu_{B}\co ((\mu_{B}\co (B\ot f))\ot (\lambda_{B}\co f))\co (B\ot \delta_{H})$$
$$=\mu_B\circ(\mu_B\ot \lambda_B)\circ (B\ot (\delta_B\co f))=B\ot (\varepsilon_{B}\co f)=B\ot \varepsilon_{H}$$

Similarly, by the same arguments,  
$$\mu_{B}\co ((\mu_{B}\co (B\ot (f\co \lambda_{H})))\ot f)\co (B\ot \delta_{H})=B\ot \varepsilon_{H},$$
and this implies that $f$ is an anchor morphism. 

As a consequence, the examples of strong projections that we can find in \cite{our1} and \cite{cocy} provide examples of anchor morphisms. 

}
\end{example}

If $(B,\rho_{B})$ is a right $H$-comodule monoid the following identity 
\begin{equation}
\label{Eti-1}
(B\ot \overline{\Pi}_{H}^{R})\co \rho_{B}\co \mu_{B}=(\mu_{B}\ot \overline{\Pi}_{H}^{R})\co (B\ot \rho_{B}).
\end{equation}
holds. Indeed:

\begin{itemize}
\item[ ]$\hspace{0.38cm}(B\ot \overline{\Pi}_{H}^{R})\co \rho_{B}\co \mu_{B}$

\item[ ]$=(\mu_{B}\ot (\varepsilon_{H}\co \mu_{H}\co (\mu_{H}\ot H))\ot H)\co (B\ot c_{H,B}\ot H\ot (\delta_{H}\co \eta_{H}))\co (\rho_{B}\ot \rho_{B}) ${\scriptsize  ({\blue (\ref{chmagma}) and definition of $\overline{\Pi}_{H}^{R}$})}

\item[ ]$=(\mu_{B}\ot (((\varepsilon_{H}\co \mu_{H})\ot (\varepsilon_{H}\co \mu_{H}))\co (H\ot \delta_{H}\ot H))\ot H)\co (B\ot c_{H,B}\ot H\ot (\delta_{H}\co \eta_{H}))\co (\rho_{B}\ot \rho_{B}) $
\item[ ]$\hspace{0.38cm}${\scriptsize ({\blue (a2) of Definition \ref{W-H-quasi}})}

\item[ ]$= (B\ot \varepsilon_{H} \ot \varepsilon_{H}\ot H)\co (\mu_{B\ot H}\ot \mu_{H}\ot H)\co (\rho_{B}\ot \rho_{B}\ot H\ot (\delta_{H}\co \eta_{H}))\co (B\ot \rho_{B})$ {\scriptsize  ({\blue comodule condition }}
\item[ ]$\hspace{0.38cm}${\scriptsize {\blue for $B$})}

\item[ ]$=(((B\ot \varepsilon_{H})\co \rho_{B}\co \mu_{B})\ot \overline{\Pi}_{H}^{R}))\co (B\ot \rho_{B}) $ {\scriptsize  ({\blue (\ref{chmagma}) and definition of $\overline{\Pi}_{H}^{R}$})}

\item[ ]$=(\mu_{B}\ot \overline{\Pi}_{H}^{R})\co (B\ot \rho_{B})$ {\scriptsize  ({\blue comodule condition for $B$}).}

\end{itemize}

As a consequence, if $h:H\rightarrow B$ is a total integral, the  equality
\begin{equation}
\label{c3}
\mu_{B}\co (\mu_{B}\ot (h\co \Pi_{H}^{R}))\co (B\ot \rho_{B})=\mu_{B}
\end{equation}
holds because 
\begin{itemize}
\item[ ]$\hspace{0.38cm}\mu_{B}\co (\mu_{B}\ot (h\co \Pi_{H}^{R}))\co (B\ot \rho_{B})$

\item[ ]$=\mu_{B}\co (\mu_{B}\ot (h\co \lambda_{H}\co \overline{\Pi}_{H}^{R}))\co (B\ot \rho_{B})$ {\scriptsize  ({\blue (40) of \cite{Asian}})}

\item[ ]$= \mu_{B}\co (\mu_{B}\ot (h\co \lambda_{H}\co \overline{\Pi}_{H}^{R}))\co \rho_{B}\co \mu_{B}$ {\scriptsize  ({\blue (\ref{Eti-1})})}

\item[ ]$= \mu_{B}\co (\mu_{B}\ot (h\co \Pi_{H}^{R}))\co \rho_{B}\co \mu_{B}$ {\scriptsize  ({\blue (40) of \cite{Asian}})}

\item[ ]$=\mu_{B}$ {\scriptsize  ({\blue (\ref{idemp-3})}).}

\end{itemize}

Also, if $h:H\rightarrow B$ is a multiplicative morphism of right $H$-comodules, the following identity holds
\begin{equation}
\label{anchor1}
h\co \Pi_{H}^{L}=q_{B}\co h, 
\end{equation}
because 
\begin{itemize}
\item[ ]$\hspace{0.38cm}h\co \Pi_{H}^{L}$

\item[ ]$= \mu_{B}\co (h\ot (h\co \lambda_{H}))\co \delta_{H}$ {\scriptsize  ({\blue $h$ is multiplicative})}

\item[ ]$=q_{B}\co h$ {\scriptsize  ({\blue $h$ is a comodule morphism}).}

\end{itemize}

If $h:H\rightarrow B$ is an anchor morphism we have that the equality
\begin{equation}
\label{c4}
\mu_{B}\co ((\mu_{B}\co (B\ot (h\co \Pi_{H}^{L})))\ot h)\co (B\ot \delta_{H})=\mu_{B}\co (B\ot h)
\end{equation}
holds. Indeed:
\begin{itemize}
\item[ ]$\hspace{0.38cm} \mu_{B}\co ((\mu_{B}\co (B\ot (h\co \Pi_{H}^{L})))\ot h)\co (B\ot \delta_{H})$

\item[ ]$= \mu_{B}\co ((\mu_{B}\co (\mu_{B}\ot (h\co \lambda_{H}))\co (B\ot \rho_{B}))\ot h)\co (B\ot (\rho_{B}\co h))$ {\scriptsize  ({\blue (\ref{re-c1}) and  comodule morphism condition }}
\item[ ]$\hspace{0.38cm}${\scriptsize {\blue (\ref{re-c1})  for $h$})}

\item[ ]$=\mu_{B}\co ((\mu_{B}\co (B\ot (h\co \lambda_{H})))\ot h)\co (B\ot \delta_{H}))\co (\mu_{B}\ot H)\co (B\ot (\rho_{B}\co h))$ {\scriptsize ({\blue comodule condition for $B$})}

\item[ ]$=\mu_{B}\co ((\mu_{B}\co (h\co  \Pi_{H}^{R}))\co (B\ot (\rho_{B}\co h))$ {\scriptsize  ({\blue  (c2) of Definition \ref{anchor}})}

\item[ ]$=\mu_{B}\co (B\ot h) $ {\scriptsize  ({\blue (\ref{c3})}).}

\end{itemize}

Then, by (\ref{anchor1}) and the comodule morphism condition for $h$, (\ref{c4}) is equivalent to 
\begin{equation}
\label{anchor2}
\mu_{B}\co ((\mu_{B}\co (B\ot q_{B}))\ot h)\co (B\ot (\rho_{B}\co h))=\mu_{B}\co (B\ot h). 
\end{equation}

Finally, for an anchor morphism $h:H\rightarrow B$ the equality 
\begin{equation}
\label{anchor21}
\mu_{B}\co (q_{B}\ot h)\co \rho_{B}=id_{B}  
\end{equation}
holds, because 
\begin{itemize}
\item[ ]$\hspace{0.38cm}\mu_{B}\co (q_{B}\ot h)\co \rho_{B}$

\item[ ]$=\mu_{B}\co ((\mu_{B}\co (B\ot (h\co \lambda_{H})))\ot h)\co (B\ot \delta_{H})\co \rho_{B} $ {\scriptsize  ({\blue comodule condition for $B$})}

\item[ ]$=\mu_{B}\co (B\ot (h\co \Pi_{H}^{R}))\co \rho_{B} $ {\scriptsize  ({\blue (c2) of Definition \ref{anchor}})}

\item[ ]$= id_{B}$ {\scriptsize  ({\blue (\ref{idemp-3})}).}

\end{itemize}

\end{definition}

\begin{definition}
\label{H-D-mod}
Let $H$ be a weak Hopf quasigroup and let $(B,\rho_{B})$ be a right $H$-comodule
magma. Let $h:H\rightarrow B$ be an anchor morphism and let $M$ be an object  in ${\mathcal
C}$. We say that $(M, \phi_{M}, \rho_{M})$ is a  strong $(H,B,h)$-Hopf module if the following axioms hold:
\begin{itemize}
\item[(d1)] The pair $(M, \rho_{M})$ is a right $H$-comodule.
\item[(d2)] The morphism $\phi_{M}:M\ot B\rightarrow M$ satisfies:
\begin{itemize}
\item[(d2-1)] $\phi_{M}\co (M\ot \eta_{B})=id_{M}.$
\item[(d2-2)] $\phi_{M}\co ((\phi_{M}\co (M\ot i_{B}))\ot B)=\phi_{M}\co (M\ot (\mu_{B}\co (i_{B}\ot B))).$
\item[(d2-3)] $\rho_{M}\co \phi_{M}=(\phi_{M}\ot \mu_{H})\co (M\ot c_{H,B}\ot H)\co (\rho_{M}\ot \rho_{B})$.
\item[(d2-4)] $\phi_{M}\circ ((\phi_{M}\co (M\ot h))\ot (h\co\lambda_H))\circ (M\ot \delta_{H})=\phi_{M}\co (M\ot (h\co \Pi_{H}^{L})).$
\item[(d2-5)] $\phi_{M}\circ ((\phi_{M}\co (M\ot (h\co \lambda_{H})))\ot h)\circ (M\ot \delta_{H})=\phi_{M}\co (M\ot (h\co \Pi_{H}^{R})).$

\end{itemize}
\end{itemize}

For example, the triple $(H, \mu_{H}, \delta_{H})$ is a strong $(H,H,id_{H})$-Hopf module. Also, if the following equality 
\begin{equation}
\label{strong-1}
\mu_{B}\co ((\mu_{B}\co (B\ot i_{B}))\ot B)=\mu_{B}\co (B\ot (\mu_{B}\co (i_{B}\ot B)))
\end{equation}
holds, the triple $(B, \mu_{B}, \rho_{B})$ is a strong $(H,B, h)$-Hopf module. 

Let $(M, \phi_{M}, \rho_{M})$ be a strong $(H,B,h)$-Hopf module.  In a similar way to (\ref{Eti-1}), but using (d2-3) of Definition \ref{H-D-mod} instead of (\ref{chmagma}), it is easy to see that 
\begin{equation}
\label{Eti-1-1}
(M\ot \overline{\Pi}_{H}^{R})\co \rho_{M}\co \phi_{M}=(\phi_{M}\ot \overline{\Pi}_{H}^{R})\co (M\ot \rho_{M})
\end{equation}
holds. Moreover,  by  (\ref{idemp-2}), (d2-3), (d2-1), and the condition of comodule for $M$, we also have the equality
\begin{equation}
\label{idemp-3M}
\phi_{M}\co (M\ot (h\co \Pi_{H}^{R}))\co \rho_{M}=id_{M}.
\end{equation}

 As a consequence, we can obtain the Hopf module version of (\ref{c3}), i.e., 
\begin{equation}
\label{d2-6}
\phi_{M}\circ (\phi_{M}\ot (h\co\Pi_{H}^{R}))\circ (M\ot \rho_{B})=\phi_{M}.
\end{equation}

Moreover, in this setting, the  equality 
\begin{equation}
\label{d2-7}\phi_{M}\circ ((\phi_{M}\co (M\ot (h\co \Pi_{H}^{L})))\ot h)\circ (M\ot \delta_{H})=\phi_{M}\co (M\ot h)
\end{equation}
holds because 
\begin{itemize}
\item[ ]$\hspace{0.38cm}\phi_{M}\circ ((\phi_{M}\co (M\ot (h\co \Pi_{H}^{L})))\ot h)\circ (M\ot \delta_{H}) $

\item[ ]$=\phi_{M}\circ ((\phi_{M}\co (M\ot (q_{B}\co h)))\ot h)\circ (M\ot \delta_{H})  $ {\scriptsize  ({\blue  (\ref{anchor1})})}

\item[ ]$= \phi_{M}\circ (M\ot (\mu_{B}\co  ((q_{B}\co h)\ot h)\co \delta_{H}))$ {\scriptsize  ({\blue  (d2-2) of Definition \ref{H-D-mod}})}

\item[ ]$=\phi_{M}\circ (M\ot (\mu_{B}\co  ((h\co \Pi_{H}^{L})\ot h)\co \delta_{H})) $ {\scriptsize  ({\blue (\ref{anchor1})})}

\item[ ]$=\phi_{M}\circ (M\ot (h\co (\Pi_{H}^{L}\ast id_{H})))  $ {\scriptsize  ({\blue $h$ is multiplicative})}

\item[ ]$=\phi_{M}\co (M\ot h) $ {\scriptsize  ({\blue (4) of \cite{Asian}}).}

\end{itemize}

Finally,  by (\ref{anchor1}) the equality (\ref{d2-7}) is equivalent to
\begin{equation} 
\label{anchor3}
\phi_{M}\circ ((\phi_{M}\co (M\ot q_{B}))\ot h)\circ (M\ot (\rho_{B}\co h))=\phi_{M}\co (M\ot h).
\end{equation}

\end{definition}

\begin{proposition}
\label{idem-M}
Let $H$ be a weak Hopf quasigroup and let $(B,\rho_{B})$ be a right $H$-comodule
magma. Let $h:H\rightarrow B$ be an anchor morphism and let $(M, \phi_{M}, \rho_{M})$ be a strong $(H,B,h)$-Hopf module. The endomorphism $q_{M}:=\phi_{M}\co (M\ot (h\co \lambda_{H}))\co \rho_{M}:M\rightarrow M$  satisfies 
\begin{equation}
\label{idemp-0M}
\rho_{M}\co q_{M}=(M\ot \Pi_{H}^{L})\co\rho_{M}\co q_{M},
\end{equation}
\begin{equation}
\label{idemp-1M}
\rho_{M}\co q_{M}=(M\ot \overline{\Pi}_{H}^{R})\co\rho_{M}\co q_{M},
\end{equation}
and, as a consequence, $q_{M}$ is an idempotent. Moreover, if $M^{co H}$ (object of coinvariants)  is the image of $q_{M}$ and $p_{M}:M\rightarrow M^{co H}$, $i_{M}:M^{co H}\rightarrow M$ are the morphisms such that $q_{M}=i_{M}\co p_{M}$ and 
 $id_{M^{co H}}=p_{M}\co i_{M}$, 
$$
\setlength{\unitlength}{3mm}
\begin{picture}(30,4)
\put(3,2){\vector(1,0){4}} \put(11,2.5){\vector(1,0){10}}
\put(11,1.5){\vector(1,0){10}} \put(1,2){\makebox(0,0){$M^{co H}$}}
\put(9,2){\makebox(0,0){$M$}} \put(24,2){\makebox(0,0){$M\ot H,$}}
\put(5.5,3){\makebox(0,0){$i_{M}$}}
\put(16,3.5){\makebox(0,0){$ \rho_{M}$}}
\put(16,0.15){\makebox(0,0){$(M\ot \overline{\Pi}_{H}^{R})\co\rho_{M}$}}
\end{picture}
$$ 
$$
\setlength{\unitlength}{3mm}
\begin{picture}(30,4)
\put(3,2){\vector(1,0){4}} \put(11,2.5){\vector(1,0){10}}
\put(11,1.5){\vector(1,0){10}} \put(1,2){\makebox(0,0){$M^{co H}$}}
\put(9,2){\makebox(0,0){$M$}} \put(24,2){\makebox(0,0){$M\ot H,$}}
\put(5.5,3){\makebox(0,0){$i_{M}$}}
\put(16,3.25){\makebox(0,0){$ \rho_{M}$}}
\put(16,0.15){\makebox(0,0){$(M\ot \Pi_{H}^{L})\co\rho_{M}$}}
\end{picture}
$$
are equalizer diagrams. 
\end{proposition}

\begin{proof}
The proof is similar to the one developed in Proposition \ref{idem-B} but using (d2-3) instead of (\ref{chmagma}) and the comodule condition for $M$ instead of the comodule condition for $B$. 
\end{proof}

\begin{remark}
Note that, as in the case of $B$, the object of coinvariants $M^{coH}$ is independent of the anchor morphism $h$ because it is the equalizer object of $\rho_{M}$ and $(M\ot \overline{\Pi}_{H}^{R})\co\rho_{M}$ (or $\rho_{M}$ and $(M\ot \Pi_{H}^{L})\co\rho_{M}$).
\end{remark}

\begin{lemma}
\label{l-idem}
Let $H$ be a weak Hopf quasigroup and let $(B,\rho_{B})$ be a right $H$-comodule
magma. Let $h:H\rightarrow B$ be an anchor morphism and let $(M, \phi_{M}, \rho_{M})$ be a  strong $(H,B,h)$-Hopf module. The following equalities hold:
\begin{equation}
\label{l-idem0}
\rho_{M}\co \phi_{M}\co (i_{M}\ot B)=(\phi_{M}\ot H)\co (i_{M}\ot \rho_{B}),
\end{equation}
\begin{equation}
\label{l-idem2}
q_{M}\co \phi_{M}\co (i_{M}\ot B)=\phi_{M}\co (i_{M}\ot q_{B}),
\end{equation}
\begin{equation}
\label{l-idem1}
(q_{M}\ot H)\co \rho_{M}\co \phi_{M}\co (i_{M}\ot B)=((\phi_{M}\co (M\ot q_{B}))\ot H)\co (i_{M}\ot \rho_{B}),
\end{equation}
\begin{equation}
\label{l-idem1-1}
(p_{M}\ot H)\co \rho_{M}\co \phi_{M}\co (i_{M}\ot B)=((p_{M}\co\phi_{M}\co (M\ot q_{B}))\ot H)\co (i_{M}\ot \rho_{B}),
\end{equation}
\begin{equation}
\label{l-idem2-1}
q_{M}\co \phi_{M}\co (i_{M}\ot i_{B})=\phi_{M}\co (i_{M}\ot i_{B}),
\end{equation}
\begin{equation}
\label{l-idem3}
p_{M}\co \phi_{M}\co (i_{M}\ot B)=p_{M}\co \phi_{M}\co (i_{M}\ot q_{B}).
\end{equation}
\begin{equation}
\label{l-idem4}
\phi_{M}\co (q_{M}\ot h)\co \rho_{M}=id_{M}.
\end{equation}

\end{lemma}

\begin{proof} The proof for the first equality is the following:
\begin{itemize}
\item[ ]$\hspace{0.38cm} \rho_{M}\co \phi_{M}\co (i_{M}\ot B)$

\item[ ]$= (\phi_{M}\ot \mu_{H})\co (M\ot c_{H,B}\ot H)\co ((\rho_{M}\co i_{M})\ot \rho_{B}) $ {\scriptsize ({\blue (d2-3) of Definition (\ref{H-D-mod})})}

\item[ ]$=(\phi_{M}\ot \mu_{H})\co (M\ot c_{H,B}\ot H)\co (((M\ot \overline{\Pi}_{H}^{R})\co\rho_{M}\co i_{M})\ot \rho_{B})  $ {\scriptsize  ({\blue (\ref{idemp-1M})})}

\item[ ]$= (\phi_{M}\ot (\mu_{H}\co ( \overline{\Pi}_{H}^{R}\ot H)))\co (M\ot c_{H,B}\ot H)\co ((\rho_{M}\co i_{M})\ot \rho_{B}) $ {\scriptsize  ({\blue naturality of $c$})}

\item[ ]$=(\phi_{M}\ot (\varepsilon_{H}\co \mu_{H})\ot H)\co (M\ot c_{H,B}\ot H\ot H)\co ((\rho_{M}\co i_{M})\ot ((B\ot \delta_{H})\co \rho_{B})) $ {\scriptsize ({\blue (10) of \cite{Asian}})}

\item[ ]$= (((\phi_{M}\ot (\varepsilon_{H}\co \mu_{H}))\co (M\ot c_{H,B}\ot H)\co (\rho_{M} \ot \rho_{B}))\ot H)\co (i_{M}\ot \rho_{B})$ {\scriptsize  ({\blue comodule condition for $B$})}

\item[ ]$= (((M\ot \varepsilon_{H})\co \rho_{M}\co \phi_{M})\ot H)\co (i_{M}\ot \rho_{B}) $ {\scriptsize  ({\blue (d2-3) of Definition (\ref{H-D-mod})})}

\item[ ]$= (\phi_{M}\ot H)\co (i_{M}\ot \rho_{B}) $ {\scriptsize  ({\blue comodule condition for $B$}).}

\end{itemize}

The equality (\ref{l-idem2}) follows by 

\begin{itemize}
\item[ ]$\hspace{0.38cm} q_{M}\co \phi_{M}\co (i_{M}\ot B)$

\item[ ]$= \phi_{M}\circ (\phi_{M}\ot (h\co\lambda_H))\circ  (i_{M}\ot \rho_{B}) $ {\scriptsize  ({\blue (\ref{l-idem0})})}

\item[ ]$=\phi_{M}\co (i_{M}\ot q_{B})$ {\scriptsize  ({\blue (d2-4) of Definition (\ref{H-D-mod})}).}

\end{itemize}

Thus, composing in (\ref{l-idem0}) with $q_{M}\ot H$, we obtain (\ref{l-idem1}). Also, composing in (\ref{l-idem1}) with $p_{M}\ot B$, we have (\ref{l-idem1-1}). Moreover, composing with $M^{coH}\ot i_{B}$ in (\ref{l-idem2}), we obtain (\ref{l-idem2-1}), and doing the same with 
$p_{M}$ we get (\ref{l-idem3}). Finally, (\ref{l-idem4}) holds because 
\begin{itemize}
\item[ ]$\hspace{0.38cm}\phi_{M}\co (q_{M}\ot h)\co \rho_{M}$

\item[ ]$=\phi_{M}\co ((\phi_{M}\co (M\ot (h\co \lambda_{H})))\ot h)\co (M\ot \delta_{H})\co \rho_{M} $ {\scriptsize  ({\blue comodule condition for $M$})}

\item[ ]$=\phi_{M}\co (M\ot (h\co \Pi_{H}^{R}))\co \rho_{M} $ {\scriptsize  ({\blue (d2-5) of Definition (\ref{H-D-mod})})}

\item[ ]$= id_{M}$ {\scriptsize  ({\blue (\ref{idemp-3M})}).}

\end{itemize}

\end{proof}

\begin{proposition}
\label{B-coinv}
Let $H$ be a weak Hopf quasigroup and let $(B,\rho_{B})$ be a right $H$-comodule
magma. Let $h:H\rightarrow B$ be an anchor morphism. If (\ref{strong-1}) holds, the submagma of coinvariants 
$(B^{co H}, \eta_{B^{co H}}, \mu_{B^{co H}})$ is a monoid. 
\end{proposition}

\begin{proof} Firstly, remember that if (\ref{strong-1}) holds, the triple $(B,\mu_{B},\rho_{B})$ is a strong $(H,B,h)$-Hopf module. Then by (\ref{l-idem2}), 
\begin{equation}
\label{l-idemB}
q_{B}\co \mu_{B}\co (i_{B}\ot i_{B})=\mu_{B}\co (i_{B}\ot i_{B}).
\end{equation}

 As a consequence, $B^{co H}$ is a monoid because, 

\begin{itemize}
\item[ ]$\hspace{0.38cm} i_{B}\co \mu_{B^{co H}} \co (B^{coH}\ot \mu_{B^{co H}})$

\item[ ]$= q_{B}\co \mu_{B}\co (i_{B}\ot (q_{B}\co \mu_{B}\co (i_{B}\ot i_{B})))$ {\scriptsize  ({\blue (\ref{muH1})})}

\item[ ]$=q_{B}\co \mu_{B}\co (i_{B}\ot ( \mu_{B}\co (i_{B}\ot i_{B})))$ {\scriptsize  ({\blue (\ref{l-idemB})})}

\item[ ]$= q_{B}\co \mu_{B}\co (( \mu_{B}\co (i_{B}\ot i_{B}))\ot i_{B}) $ {\scriptsize  ({\blue (\ref{strong-1})})}

\item[ ]$=q_{B}\co \mu_{B}\co (( q_{B}\co \mu_{B}\co (i_{B}\ot i_{B}))\ot i_{B})$ {\scriptsize  ({\blue (\ref{l-idemB})})}

\item[ ]$=i_{B}\co \mu_{B^{co H}} \co (\mu_{B^{co H}}\ot B^{co H})$ {\scriptsize  ({\blue (\ref{muH1})}).}

\end{itemize}

\end{proof}

\begin{proposition}
\label{M-coinv}
Let $H$ be a weak Hopf quasigroup and let $(B,\rho_{B})$ be a right $H$-comodule
magma. Let $h:H\rightarrow B$ be an anchor morphism. If (\ref{strong-1}) holds, for all strong $(H,B,h)$-Hopf module $(M, \phi_{M}, \rho_{M})$,  the object of coinvariants $M^{coH}$ is a right $B^{coH}$-module. 
\end{proposition}

\begin{proof} First note that, by Proposition \ref{B-coinv}, the object $B^{co H}$ is a monoid. Now we will show that there exists  an action $\phi_{M^{coH}}:M^{coH}\ot B^{coH}\rightarrow M^{coH}$ such that $\phi_{M^{coH}}\co (M^{coH}\ot \eta_{B^{co H}})=id_{M^{coH}},$ and $\phi_{M^{coH}}\co (\phi_{M^{coH}}\ot B^{co H})=\phi_{M^{coH}}\co (M^{coH}\ot \mu_{B^{co H}})$. To define the action, we begin by proving that 
\begin{equation}
\label{phiMcoH}
\rho_{M}\co \phi_{M}\co (i_{M}\ot i_{B})=(M\ot \Pi_{H}^{L})\co \rho_{M}\co \phi_{M}\co (i_{M}\ot i_{B}).
\end{equation}

Indeed, 

\begin{itemize}
\item[ ]$\hspace{0.38cm} \rho_{M}\co \phi_{M}\co (i_{M}\ot i_{B})$

\item[ ]$=(\phi_{M}\ot H)\co (i_{M}\ot (\rho_{B}\co i_{B}))$ {\scriptsize  ({\blue (\ref{l-idem0})})}

\item[ ]$= (\phi_{M}\ot \Pi_{H}^{L})\co (i_{M}\ot (\rho_{B}\co i_{B}))$ {\scriptsize  ({\blue (\ref{idemp-0})})}

\item[ ]$= (M\ot \Pi_{H}^{L})\co \rho_{M}\co \phi_{M}\co (i_{M}\ot i_{B})$ {\scriptsize  ({\blue (\ref{l-idem0})})}

\end{itemize}

Then, there exists a unique morphism $\phi_{M^{coH}}:M^{coH}\ot B^{coH}\rightarrow M^{coH}$ such that
\begin{equation}
\label{PhiBcoH}
\phi_{M}\co (i_{M}\ot i_{B})=i_{M}\co \phi_{M^{coH}}.
\end{equation}

Therefore, 
\begin{equation}
\label{PhiBcoH1}
\phi_{M^{coH}}=p_{M}\co \phi_{M}\co (i_{M}\ot i_{B}).
\end{equation}

The pair $(M^{coH}, \phi_{M^{coH}})$ satisfies the conditions of right $B^{coH}$-module because 
$$\phi_{M^{coH}}\co (M^{coH}\ot \eta_{B^{co H}})\stackrel{{\scriptsize \blue  (\ref{etaH})}}{=}
p_{M}\co \phi_{M}\co (i_{M}\ot \eta_{B})\stackrel{{\scriptsize \blue  ({\rm d2-1})}}{=}p_{M}\co i_{M}=id_{M^{coH}},$$
and 
\begin{itemize}
\item[ ]$\hspace{0.38cm}\phi_{M^{coH}}\co (\phi_{M^{coH}}\ot B^{co H})$

\item[ ]$=p_{M}\co \phi_{M}\co  ((q_{M}\co \phi_{M}\co (i_{M}\ot i_{B}))\ot i_{B})$ {\scriptsize  ({\blue (\ref{PhiBcoH1})})}

\item[ ]$= p_{M}\co \phi_{M}\co  ((\phi_{M}\co (i_{M}\ot i_{B}))\ot i_{B})$ {\scriptsize  ({\blue (\ref{l-idem2-1})})}

\item[ ]$=p_{M}\co \phi_{M}\co  (i_{M}\ot (\mu_{B}\co (i_{B}\ot i_{B})))  $ {\scriptsize  ({\blue (d2-2) of Definition (\ref{H-D-mod})})}

\item[ ]$=p_{M}\co \phi_{M}\co  (i_{M}\ot (i_{B}\co \mu_{B^{coH}})) $ {\scriptsize  ({\blue (\ref{muH})})}

\item[ ]$=\phi_{M^{coH}}\co (M^{coH}\ot \mu_{B^{co H}}) $ {\scriptsize  ({\blue (\ref{PhiBcoH1})})}.

\end{itemize}

\end{proof}

\begin{proposition}
\label{tensor}
Let $H$ be a weak Hopf quasigroup and let $(B,\rho_{B})$ be a right $H$-comodule
magma. Let $h:H\rightarrow B$ be an anchor morphism. Assume that (\ref{strong-1})  and 
\begin{equation}
\label{strong2}
\mu_{B}\co (i_{B}\ot \mu_{B})=\mu_{B}\co ((\mu_{B}\co (i_{B}\ot B))\ot B)
\end{equation}
hold. Then if the category ${\mathcal C}$ admits coequalizers and the functors $-\ot B$ and $-\ot H$ preserve coequalizers, for all strong $(H,B,h)$-Hopf module $(M, \phi_{M}, \rho_{M})$,  the object $M^{coH}\ot_{B^{coH}}B$, defined by the coequalizer of $T_{M}^{1}=\phi_{M^{coH}}\ot B$ and $T_{M}^{2}=M^{coH}\ot (\mu_{B}\co (i_{B}\ot B))$, is a strong $(H,B,h)$-Hopf module. Moreover, 
$M^{coH}\ot_{B^{coH}}B$ and $M$ are isomorphic as right $H$-comodules.
\end{proposition}

\begin{proof} In the first step, we begin by proving the existence of an action and a coaction for $M^{coH}\ot_{B^{coH}}H$.  If the object $M^{coH}\ot_{B^{coH}}H$ is defined by the coequalizer diagram 

\begin{equation}
\label{coeq-1}
\setlength{\unitlength}{1mm}
\begin{picture}(101.00,10.00)
\put(22.50,8.00){\vector(1,0){20.00}}
\put(22.50,4.00){\vector(1,0){20.00}}
\put(63.00,6.00){\vector(1,0){20.00}}
\put(32.00,11.00){\makebox(0,0)[cc]{$T_{M}^{1}$ }}
\put(32.00,1.00){\makebox(0,0)[cc]{$T_{M}^{2}$}} 
\put(74.00,9.00){\makebox(0,0)[cc]{$n_{M^{coH}}$ }}
\put(7.00,6.00){\makebox(0,0)[cc]{$ M^{coH}\otimes B^{coH}\ot B$ }}
\put(54.00,6.00){\makebox(0,0)[cc]{$ M^{coH}\ot B$ }}
\put(96.00,6.00){\makebox(0,0)[cc]{$M^{coH}\ot_{B^{coH}} B, $ }}
\end{picture}
\end{equation}
we have that 
\begin{equation}
\label{coeq-2}
\setlength{\unitlength}{1mm}
\begin{picture}(101.00,10.00)
\put(19.50,8.00){\vector(1,0){20.00}}
\put(19.50,4.00){\vector(1,0){20.00}}
\put(67.00,6.00){\vector(1,0){15.00}}
\put(30.00,11.00){\makebox(0,0)[cc]{$T_{M}^{1}\ot B$ }}
\put(30.00,1.00){\makebox(0,0)[cc]{$T_{M}^{2}\ot B$}} 
\put(74.00,9.00){\makebox(0,0)[cc]{$n_{M^{coH}}\ot B$ }}
\put(0.00,6.00){\makebox(0,0)[cc]{$ M^{coH}\otimes B^{coH}\ot B\ot B$ }}
\put(54.00,6.00){\makebox(0,0)[cc]{$ M^{coH}\ot B\ot B$ }}
\put(101.00,6.00){\makebox(0,0)[cc]{$M^{coH}\ot_{B^{coH}} B \ot B$ }}
\end{picture}
\end{equation}
is also a coequalizer diagram because the functor $-\ot B$ preserves coequalizers.  Consider the morphism $n_{M^{coH}}\co (M^{coH}\ot \mu_{B}): M^{coH}\ot B\ot B\rightarrow M^{coH}\ot_{B^{coH}} B$. Then, 
$$n_{M^{coH}}\co (\phi_{M^{coH}}\ot \mu_{B})\stackrel{{\scriptsize \blue  (\ref{coeq-1})}}{=}
n_{M^{coH}}\co (M^{coH}\ot (\mu_{B}\co (i_{B}\ot\mu_{B})))\stackrel{{\scriptsize \blue  (\ref{strong2})}}{=}n_{M^{coH}}\co (M^{coH}\ot (\mu_{B}\co ((\mu_{B}\co (i_{B}\ot B))\ot B))),$$
and, as a consequence, there exists a unique morphism $$\phi_{M^{coH}\ot_{B^{coH}} B}:M^{coH}\ot_{B^{coH}} B \ot B\rightarrow M^{coH}\ot_{B^{coH}} B$$ such that 
\begin{equation}
\label{phi-ten}
\phi_{M^{coH}\ot_{B^{coH}} B}\co (n_{M^{coH}}\ot B)=n_{M^{coH}}\co (M^{co H}\ot \mu_{B}).
\end{equation}

On the other hand, consider the morphism 
$$(n_{M^{coH}}\ot H)\co (M^{co H}\ot \rho_{B}): M^{coH}\ot B\rightarrow M^{coH}\ot_{B^{coH}} B\ot H.$$

Then, taking into account that, by (\ref{strong-1}), $(B,\mu_{B},\rho_{B})$ is a strong $(H,B,h)$-Hopf module, 

\begin{itemize}
\item[ ]$\hspace{0.38cm} (n_{M^{coH}}\ot H)\co (\phi_{M^{co H}}\ot \rho_{B})$

\item[ ]$= (n_{M^{coH}}\ot H)\co (M^{coH}\ot ((\mu_{B}\ot H)\co (i_{B}\ot \rho_{B})))$ {\scriptsize  ({\blue (\ref{coeq-1})})}

\item[ ]$= (n_{M^{coH}}\ot H)\co  (M^{coH}\ot (\rho_{B}\co \mu_{B}\co (i_{B}\ot B))) $ {\scriptsize  ({\blue (\ref{l-idem0})})}.

\end{itemize}

Thus, there exists a unique morphism 
$$\rho_{M^{coH}\ot_{B^{coH}} B}:M^{coH}\ot_{B^{coH}} B \rightarrow M^{coH}\ot_{B^{coH}} B\ot H$$ such that 
\begin{equation}
\label{rho-ten}
\rho_{M^{coH}\ot_{B^{coH}} B}\co n_{M^{coH}}=(n_{M^{coH}}\ot H)\co (M^{co H}\ot \rho_{B}).
\end{equation}

We proceed to show that $(M^{coH}\ot_{B^{coH}} B, \phi_{M^{coH}\ot_{B^{coH}} B}, \rho_{M^{coH}\ot_{B^{coH}} B})$ is a strong $(H,B,h)$-Hopf module. Indeed, first note that by (\ref{rho-ten}) and the comodule condition for $B$,  
$$(M^{coH}\ot_{B^{coH}} B\ot \varepsilon_{H})\co \rho_{M^{coH}\ot_{B^{coH}} B}\co n_{M^{coH}}=
(n_{M^{coH}}\ot \varepsilon_{H})\co (M^{co H}\ot \rho_{B})=n_{M^{coH}}.$$

As a consequence, and using that $n_{M^{coH}}$ is a coequalizer, we obtain that 
$$(M^{coH}\ot_{B^{coH}} B\ot \varepsilon_{H})\co \rho_{M^{coH}\ot_{B^{coH}} B}=id_{M^{coH}\ot_{B^{coH}} B}.$$

By similar arguments we have 
$$(\rho_{M^{coH}\ot_{B^{coH}} B}\ot H)\co \rho_{M^{coH}\ot_{B^{coH}} B}\co n_{M^{coH}}=
(n_{M^{coH}}\ot H\ot H)\co (M^{co H}\ot ((\rho_{B}\ot H)\co \rho_{B}))$$
$$=(n_{M^{coH}}\ot H\ot H)\co (M^{co H}\ot ((B\ot \delta_{H})\co \rho_{B}))=(M^{coH}\ot_{B^{coH}} B\ot \delta_{H})\co \rho_{M^{coH}\ot_{B^{coH}} B}\co n_{M^{coH}}.$$

Therefore, the pair $(M^{coH}\ot_{B^{coH}} B,  \rho_{M^{coH}\ot_{B^{coH}} B})$ is a right $H$-comodule and we have (d1) of Definition \ref{H-D-mod}.  

On the other hand,  by (\ref{phi-ten}), 
$$\phi_{M^{coH}\ot_{B^{coH}} B}\co (n_{M^{coH}}\ot \eta_{B})=n_{M^{coH}}, $$ 
and then, (d2-1) of Definition \ref{H-D-mod}, i.e., $\phi_{M^{coH}\ot_{B^{coH}} B}\co (M^{coH}\ot_{B^{coH}} B\ot \eta_{B})=id_{M^{coH}\ot_{B^{coH}} B}$, holds. Moreover, 
\begin{itemize}
\item[ ]$\hspace{0.38cm} \phi_{M^{coH}\ot_{B^{coH}} B}\co ((\phi_{M^{coH}\ot_{B^{coH}} B}\co (n_{M^{coH}}\ot q_{B}))\ot B)$

\item[ ]$= n_{M^{coH}}\co (M^{coH}\ot (\mu_{B}\co ((\mu_{B}\co (B\ot q_{B}))\ot B)))$ {\scriptsize  ({\blue (\ref{phi-ten})})}

\item[ ]$=n_{M^{coH}}\co (M^{coH}\ot (\mu_{B}\co (B\ot (\mu_{B}\co (q_{B}\ot B)))))$ {\scriptsize  ({\blue (\ref{strong-1})})}

\item[ ]$=\phi_{M^{coH}\ot_{B^{coH}} B}\co (n_{M^{coH}}\ot (\mu_{B}\co (q_{B}\ot B)))$ {\scriptsize  ({\blue (\ref{phi-ten})})}
\end{itemize}
and, as a consequence, 
$$\phi_{M^{coH}\ot_{B^{coH}} B}\co ((\phi_{M^{coH}\ot_{B^{coH}} B}\co (M^{coH}\ot_{B^{coH}} B\ot q_{B}))\ot B)$$ $$=\phi_{M^{coH}\ot_{B^{coH}} B}\co (M^{coH}\ot_{B^{coH}} B\ot (\mu_{B}\co (q_{B}\ot B)))$$
because $n_{M^{coH}}\ot B\ot B$ is a coequalizer. Therefore (d2-2) of Definition \ref{H-D-mod} follows because $p_{B}$ is a projection.

The proof for (d2-3) of Definition \ref{H-D-mod} is the following: Composing with the coequalizer $n_{M^{coH}}\ot B$ we obtain  that 
\begin{itemize}
\item[ ]$\hspace{0.38cm} \rho_{M^{coH}\ot_{B^{coH}} B}\co \phi_{M^{coH}\ot_{B^{coH}} B}\co  (n_{M^{coH}}\ot  B)$

\item[ ]$=  (n_{M^{coH}}\ot  H)\co (M^{coH}\ot (\rho_{B}\co \mu_{B}))$ {\scriptsize  ({\blue (\ref{phi-ten}), (\ref{rho-ten})})}

\item[ ]$=(n_{M^{coH}}\ot  H)\co (M^{coH}\ot (\mu_{B\ot H}\co (\rho_{B}\ot \rho_{B}))) $ {\scriptsize  ({\blue (\ref{chmagma})})}

\item[ ]$= (\phi_{M^{coH}\ot_{B^{coH}} B}\ot \mu_{H})\co (M^{coH}\ot_{B^{coH}} B\ot c_{H,B}\ot H)\co ((\rho_{M^{coH}\ot_{B^{coH}} B}\co n_{M^{coH}})\ot \rho_{B})$ {\scriptsize  ({\blue (\ref{phi-ten}), (\ref{rho-ten})})}

\end{itemize}
and thus  
$$\rho_{M^{coH}\ot_{B^{coH}} B}\co \phi_{M^{coH}\ot_{B^{coH}} B}=(\phi_{M^{coH}\ot_{B^{coH}} B}\ot \mu_{H})\co (M^{coH}\ot_{B^{coH}} B\ot c_{H,B}\ot H)\co (\rho_{M^{coH}\ot_{B^{coH}} B}\ot \rho_{B}).$$

Using that $-\ot H$ preserves coequalizers, the equality (d2-4) of Definition \ref{H-D-mod} follows from 
\begin{itemize}
\item[ ]$\hspace{0.38cm} \phi_{M^{coH}\ot_{B^{coH}} B}\circ (\phi_{M^{coH}\ot_{B^{coH}} B}\ot B)\co 
(n_{M^{coH}}\ot ((h\ot (h\co \lambda_{H}))\co \delta_{H}))$

\item[ ]$=  n_{M^{coH}}\co (M^{coH}\ot (\mu_{B}\circ (\mu_{B}\ot B)\co (B\ot ((h\ot (h\co \lambda_{H}))\co \delta_{H}))))$ {\scriptsize  ({\blue (\ref{phi-ten})})}

\item[ ]$=n_{M^{coH}}\co (M^{coH}\ot (\mu_{B}\co (B\ot (h\co \Pi_{H}^{L}))))$ {\scriptsize  ({\blue (c1) of Definition \ref{anchor}})}

\item[ ]$= \phi_{M^{coH}\ot_{B^{coH}} B}\co (n_{M^{coH}}\ot (h\co \Pi_{H}^{L}))$ {\scriptsize  ({\blue (\ref{phi-ten})})}.

\end{itemize}

Also, we have 
\begin{itemize}
\item[ ]$\hspace{0.38cm} \phi_{M^{coH}\ot_{B^{coH}} B}\circ ((\phi_{M^{coH}\ot_{B^{coH}} B}\co (M^{coH}\ot_{B^{coH}} B\ot (h\co \lambda_{H})))\ot h)\circ (n_{M^{coH}}\ot \delta_{H})$

\item[ ]$=  n_{M^{coH}}\co (M^{coH}\ot (\mu_{B}\circ ((\mu_{B}\co (B\ot (h\co \lambda_{H})))\ot h)\circ (B\ot \delta_{H})))$ {\scriptsize  ({\blue (\ref{phi-ten})})}

\item[ ]$=n_{M^{coH}}\co (M^{coH}\ot (\mu_{B}\co (B\ot (h\co \Pi_{H}^{R}))))$ {\scriptsize  ({\blue (c2) of Definition \ref{anchor}})}

\item[ ]$= \phi_{M^{coH}\ot_{B^{coH}} B}\co (n_{M^{coH}}\ot (h\co \Pi_{H}^{R}))$ {\scriptsize  ({\blue (\ref{phi-ten})})}, 

\end{itemize}
 and then (d2-5) of Definition \ref{H-D-mod} holds. 

Consider the morphism $\phi_{M}\co (i_{M}\ot B):M^{coH}\ot B\rightarrow M$. By (\ref{PhiBcoH}) and (d2-2) of Definition \ref{H-D-mod}, we obtain that $\phi_{M}\co (i_{M}\ot B)\co T_{M}^{1}=\phi_{M}\co (i_{M}\ot B)\co T_{M}^{2}$ and, as a consequence, there exists a unique morphism $\omega_{M}:M^{coH}\ot_{B^{coH}} B\rightarrow M$ such that 
\begin{equation}
\label{omeg}
\omega_{M}\co n_{M^{coH}}=\phi_{M}\co (i_{M}\ot B).
\end{equation}

The morphism $\omega_{M}$ is a morphism of right $H$-comodules because 
$$(\omega_{M}\ot H)\co\rho_{M^{coH}\ot_{B^{coH}} B}\co n_{M^{coH}}\stackrel{{\scriptsize \blue  (\ref{rho-ten})}}{=}((\omega_{M}\co n_{M^{coH}})\ot H)\co (M^{coH}\ot \rho_{B})
\stackrel{{\scriptsize \blue  (\ref{omeg})}}{=}(\phi_{M}\ot H)\co (i_{M}\ot \rho_{B})$$$$\stackrel{{\scriptsize \blue  (\ref{l-idem0})}}{=}\rho_{M}\co \phi_{M}\co (i_{M}\ot B)\stackrel{{\scriptsize \blue  (\ref{omeg})}}{=}\rho_{M}\co \omega_{M}\co n_{M^{coH}}$$
and then $(\omega_{M}\ot H)\co\rho_{M^{coH}\ot_{B^{coH}} B}=\rho_{M}\co \omega_{M}.$

Finally, $\omega_{M}$ is an isomorphism with inverse $\omega_{M}^{\prime}=n_{M^{coH}}\co (p_{M}\ot h)\co \rho_{M}$. Indeed, firstly note that, by (\ref{l-idem4}), we have that 
$\omega_{M}\co \omega_{M}^{\prime}=id_{M}.$ On the other hand, composing with $n_{M^{coH}}$ we have 

\begin{itemize}
\item[ ]$\hspace{0.38cm}\omega_{M}^{\prime} \co \omega_{M} \circ n_{M^{coH}}$

\item[ ]$= n_{M^{coH}}\co (p_{M}\ot h)\co \rho_{M}\co \phi_{M}\co (i_{M}\ot B) $ {\scriptsize  ({\blue (\ref{omeg})})}

\item[ ]$=  n_{M^{coH}}\co ((p_{M}\co \phi_{M}\co (M\ot q_{B}))\ot h)\co (i_{M}\ot \rho_{B})$ {\scriptsize  
({\blue (\ref{l-idem1-1})})}

\item[ ]$=n_{M^{coH}}\co (( \phi_{M^{coH}}\co (M^{coH}\ot p_{B}))\ot h)\co (M^{coH}\ot \rho_{B})  $ {\scriptsize  ({\blue (\ref{PhiBcoH1})})} 

\item[ ]$= n_{M^{coH}}\co (M^{coH}\ot (\mu_{B}\co (q_{B}\ot h)\co \rho_{B})) $ {\scriptsize  ({\blue (\ref{coeq-1})})}

\item[ ]$= n_{M^{coH}} $ {\scriptsize  ({\blue (\ref{anchor21})})}, 

\end{itemize}

and, as a consequence, $\omega_{M}^{\prime} \co \omega_{M}=id_{M^{coH}\ot_{B^{coH}} B}$.

\end{proof}

\begin{remark}
Note that in the previous proposition the existence of the comodule structure on $M^{coH}\ot_{B^{coH}} B$ does not depend on the preservation of coequalizers by the functors $-\ot B$ and $-\ot H$.

\end{remark}

\begin{proposition}
\label{isomorp}
Let $H$ be a weak Hopf quasigroup and let $(B,\rho_{B})$ be a right $H$-comodule
magma. Let $h:H\rightarrow B$ be an anchor morphism and let $(P, \phi_{P}, \rho_{P})$, $(Q, \phi_{Q}, \rho_{Q})$ be strong $(H,B,h)$-Hopf modules. If there exists a right $H$-comodule isomorphism $\omega :Q\rightarrow P$, the triple $(P,\phi_{P}^{\omega}=\omega\co \phi_{Q}\co (\omega^{-1}\ot B), \rho_{P})$,  called the $\omega$-deformation of  $(P, \phi_{P}, \rho_{P})$, is a strong $(H,B,h)$-Hopf module.
\end{proposition}

\begin{proof} The proof follows easily because, if $\omega$ is a right $H$-comodule isomorphism,  $\rho_{P}=(\omega\ot H)\co \rho_{Q}\co \omega^{-1}$. 
\end{proof}

\begin{definition}
Let $H$ be a weak Hopf quasigroup and let $(B,\rho_{B})$ be a right $H$-comodule
magma. Let $h:H\rightarrow B$ be an anchor morphism. Assume that (\ref{strong-1})  and (\ref{strong2})
hold. Then if the category ${\mathcal C}$ admits coequalizers and the functors $-\ot B$ and $-\ot H$ preserve coequalizers, we define the category of strong $(H,B,h)$-Hopf modules  as the one whose objects are strong $(H,B,h)$-Hopf modules, and whose morphisms $f:M\rightarrow N$ are morphisms of right $H$-comodules and 
 $B$-quasilinear, i.e. 
\begin{equation}
\label{quasilineal}
\phi_{N}^{\omega_{N}}\co (f\ot B)=f\co \phi_{M}^{\omega_{M}},
\end{equation}
where $\omega_{M}:M^{coH}\ot_{B^{coH}} B\rightarrow M$, $\omega_{N}:N^{coH}\ot_{B^{coH}} B\rightarrow N$ are the isomorphisms of right $H$-comodules obtained in the proof of Proposition \ref{tensor}. This  category   will be  denoted by ${\mathcal S}{\mathcal M}^{H}_{B}(h)$. 
\end{definition}

\begin{example}
As  particular instances of ${\mathcal S}{\mathcal M}^{H}_{B}(h)$ we will describe in detail some  interesting examples of Hopf module categories associated  to Hopf algebras, weak Hopf algebras, Hopf quasigroups and weak Hopf quasigroups. 

\underline{Examples for Hopf algebras}: If $H$ is a Hopf algebra and $B=H$, $\rho_{B}=\delta_{H}$, $h=id_{H}$, the equalities (\ref{strong-1})  and (\ref{strong2}) hold trivially because the product on $H$ is associative. In this case the category ${\mathcal S}{\mathcal M}^{H}_{H}(id_{H})$ is the one whose objects are triples $(M, \phi_{M}, \rho_{M})$ where:
\begin{itemize}
\item[(e1)] The pair $(M, \rho_{M})$ is a right $H$-comodule.
\item[(e2)] The morphism $\phi_{M}:M\ot H\rightarrow M$ satisfies:
\begin{itemize}
\item[(e2-1)] $\phi_{M}\co (M\ot \eta_{H})=id_{M},$
\item[(e2-2)] $\rho_{M}\co \phi_{M}=(\phi_{M}\ot \mu_{H})\co (M\ot c_{H,H}\ot H)\co (\rho_{M}\ot \delta_{H})$,
\item[(e2-3)] $\phi_{M}\circ (\phi_{M}\ot \lambda_H))\circ (M\ot \delta_{H})=M\ot \varepsilon_{H},$
\item[(e2-4)] $\phi_{M}\circ ((\phi_{M}\co (M\ot \lambda_{H}))\ot H)\circ (M\ot \delta_{H})=M\ot \varepsilon_{H},$
\end{itemize}
\end{itemize}
because in this case $\Pi_{H}^{L}=\Pi_{H}^{R}=q_{H}=\eta_{H}\ot \varepsilon_{H}$, $H^{coH}=K$, and $i_{H}=\eta_{H}$. Then, if $(M, \phi_{M}, \rho_{M})$ is a classical Hopf module in the sense of Larson and Sweedler \cite{Lar-Sweed} (see also \cite{Sweedler}), we have that $(M, \phi_{M}, \rho_{M})$ is an object in  ${\mathcal S}{\mathcal M}^{H}_{H}(id_{H})$ because the identity 
\begin{equation}
\label{LS-asocH}
\phi_{M}\co (\phi_{M}\ot H)=\phi_{M}\co (M\ot \mu_{H})
\end{equation}
holds. Moreover, the morphisms of ${\mathcal S}{\mathcal M}^{H}_{H}(id_{H})$ are morphisms of right $H$-comodules and  $H$-quasilinear, i.e., satisfying (\ref{quasilineal}), where $\omega_{M}=\phi_{M}\co (i_{M}\ot H):M^{co H}\ot H\rightarrow M$ is the associated isomorphism with inverse $\omega_{M}^{-1}=(p_{M}\ot H)\co \rho_{M}$, and $\phi_{M^{coH}\ot H}=M^{coH}\ot \mu_{H}$ and $\rho_{M^{coH}\ot H}=M^{co H}\ot \delta_{H}$. Then any morphism $f:M\rightarrow N$ of right $H$-comodules and linear in the classical sense, i.e., such that $f\co \phi_{M}=\phi_{N}\co (f\ot H)$, satisfies (\ref{quasilineal}). Therefore, the Larson-Sweedler category of Hopf modules, denoted by ${\mathcal M}^{H}_{H}$, is a subcategory of ${\mathcal S}{\mathcal M}^{H}_{H}(id_{H})$. 

Also, in the Hopf algebra setting we have the following more general example: Let $(B,\rho_{B})$ be a right $H$-comodule monoid such that the functors $-\ot B$, $-\ot H$ preserve coequalizers. In this case, the equalities (\ref{strong-1})  and (\ref{strong2}) hold trivially because the product on $B$ is associative. For any $h$ multiplicative total integral, therefore an anchor morphism,  the category ${\mathcal S}{\mathcal M}^{H}_{B}(h)$ is the one whose objects are triples $(M, \phi_{M}, \rho_{M})$ where:
\begin{itemize}
\item[(f1)] The pair $(M, \rho_{M})$ is a right $H$-comodule.
\item[(f2)] The morphism $\phi_{M}:M\ot B\rightarrow M$ satisfies:
\begin{itemize}
\item[(f2-1)] $\phi_{M}\co (M\ot \eta_{B})=id_{M},$
\item[(f2-2)] $\phi_{M}\co ((\phi_{M}\co (M\ot i_{B}))\ot B)=\phi_{M}\co (M\ot (\mu_{B}\co (i_{B}\ot B))),$
\item[(f2-3)] $\rho_{M}\co \phi_{M}=(\phi_{M}\ot \mu_{H})\co (M\ot c_{H,B}\ot H)\co (\rho_{M}\ot \rho_{B})$,
\item[(f2-4)] $\phi_{M}\circ ((\phi_{M}\co (M\ot h))\ot (h\co \lambda_{H}))\co (M\ot \delta_{H})
=M\ot \varepsilon_{H},$
\item[(f2-5)] $\phi_{M}\circ ((\phi_{M}\co (M\ot (h\co \lambda_{H})))\ot h)\circ (M\ot \delta_{H})=M\ot \varepsilon_{H},$
\end{itemize}
\end{itemize}
because in this case $\Pi_{H}^{L}=\Pi_{H}^{R}=\eta_{H}\ot \varepsilon_{H}$. The morphisms in ${\mathcal S}{\mathcal M}^{H}_{B}(h)$ are morphisms of right $H$-comodules and  $B$-quasilinear where $\omega_{M}:M^{\co H}\ot_{B^{coH}} H\rightarrow M$ is the associated isomorphism of right $H$-comodules defined in the proof of Proposition \ref{tensor}. Then, if $(M, \phi_{M}, \rho_{M})$ is an $(H,B)$-Hopf module in the sense of Doi \cite{Doi83}, we have that $(M, \phi_{M}, \rho_{M})$ is an object in  ${\mathcal S}{\mathcal M}^{H}_{B}(h)$ because the identity 
\begin{equation}
\label{LS-asocB}
\phi_{M}\co (\phi_{M}\ot B)=\phi_{M}\co (M\ot \mu_{B})
\end{equation}
holds. The morphisms in  ${\mathcal S}{\mathcal M}^{H}_{B}(h)$ are colinear morphisms satisfying (\ref{quasilineal}), for the action $\phi_{M^{\co H}\ot_{B^{coH}} H}$ and the coaction $\rho_{M^{\co H}\ot_{B^{coH}} H}$ defined in the proof of Proposition \ref{tensor}. Then any morphism $f:M\rightarrow N$ of right $H$-comodules and right $B$-linear, i.e., such that $f\co \phi_{M}=\phi_{N}\co (f\ot B)$, satisfies (\ref{quasilineal}). Therefore, the category of right $(H,B)$-Hopf modules, denoted by ${\mathcal M}^{H}_{B}$, is a subcategory of ${\mathcal S}{\mathcal M}^{H}_{B}(h)$ (for all  multiplicative total integral $h$).

\underline{Examples for weak Hopf algebras}: Let $H$ be a weak Hopf algebra. Let $(B,\rho_{B})$ be a right $H$-comodule monoid such that the functors $-\ot B$, $-\ot H$ preserve coequalizers. As in the Hopf algebra setting, the equalities (\ref{strong-1})  and (\ref{strong2}) hold trivially because the product on $B$ is associative. Also, in this case,  a multiplicative total integral $h$ is  an anchor morphism because the product on $B$ is associative,  (\ref{idemp-2}) holds and $\Pi_{H}^{L}\ast id_{H}=id_{H}$.  Then, the category ${\mathcal S}{\mathcal M}^{H}_{B}(h)$ is the one whose objects are triples $(M, \phi_{M}, \rho_{M})$ where:
\begin{itemize}
\item[(g1)] The pair $(M, \rho_{M})$ is a right $H$-comodule.
\item[(g2)] The morphism $\phi_{M}:M\ot B\rightarrow M$ satisfies:
\begin{itemize}
\item[(g2-1)] $\phi_{M}\co (M\ot \eta_{B})=id_{M},$
\item[(g2-2)] $\phi_{M}\co ((\phi_{M}\co (M\ot i_{B}))\ot B)=\phi_{M}\co (M\ot (\mu_{B}\co (i_{B}\ot B))),$
\item[(g2-3)] $\rho_{M}\co \phi_{M}=(\phi_{M}\ot \mu_{H})\co (M\ot c_{H,B}\ot H)\co (\rho_{M}\ot \rho_{B})$,
\item[(g2-4)] $\phi_{M}\circ ((\phi_{M}\co (M\ot h))\ot (h\co \lambda_{H}))\co (M\ot \delta_{H})=\phi_{M}\co (M\ot (h\co \Pi_{H}^{L})),$
\item[(g2-5)] $\phi_{M}\circ ((\phi_{M}\co (M\ot (h\co \lambda_{H})))\ot h)\circ (M\ot \delta_{H})=\phi_{M}\co (M\ot (h\co \Pi_{H}^{R})).$
\end{itemize}
\end{itemize}

Then, if $(M, \phi_{M}, \rho_{M})$ is an $(H,B)$-Hopf module in the sense B\"{o}hm \cite{bohm2} (see also \cite{ZZ}), the triple $(M, \phi_{M}, \rho_{M})$ is an object in  ${\mathcal S}{\mathcal M}^{H}_{B}(h)$ because the identity (\ref{LS-asocB}) holds. 

As in the Hopf case, the morphisms in  ${\mathcal S}{\mathcal M}^{H}_{B}(h)$ are colinear morphisms satisfying (\ref{quasilineal}), for $\phi_{M^{\co H}\ot_{B^{coH}} H}$ and $\rho_{M^{\co H}\ot_{B^{coH}} H}$ the action and the coaction defined in the proof of Proposition \ref{tensor}. Then any morphism $f:M\rightarrow N$ of right $H$-comodules and right $B$-linear satisfies (\ref{quasilineal}). Therefore, as in the classical context, the category of right $(H,B)$-Hopf modules, denoted by ${\mathcal M}^{H}_{B}$, is a subcategory of ${\mathcal S}{\mathcal M}^{H}_{B}(h)$ (for all  multiplicative total integral $h$).  As a consequence, if $H=B$, $\rho_{B}=\delta_{H}$ and $h=id_{H}$, we obtain that the category ${\mathcal M}^{H}_{H}$, i.e., the category of Hopf modules associated to $H$,  is a subcategory of ${\mathcal S}{\mathcal M}^{H}_{H}(id_{H})$. Note that in this case the objects of ${\mathcal S}{\mathcal M}^{H}_{H}(id_{H})$ are triples 
$(M, \phi_{M}, \rho_{M})$ where:
\begin{itemize}
\item[(h1)] The pair $(M, \rho_{M})$ is a right $H$-comodule.
\item[(h2)] The morphism $\phi_{M}:M\ot H\rightarrow M$ satisfies:
\begin{itemize}
\item[(h2-1)] $\phi_{M}\co (M\ot \eta_{H})=id_{M},$
\item[(h2-2)] $\phi_{M}\co ((\phi_{M}\co (M\ot i_{L}))\ot H)=\phi_{M}\co (M\ot (\mu_{H}\co (i_{L}\ot H))),$
\item[(h2-3)] $\rho_{M}\co \phi_{M}=(\phi_{M}\ot \mu_{H})\co (M\ot c_{H,H}\ot H)\co (\rho_{M}\ot \delta_{H})$,
\item[(h2-4)] $\phi_{M}\circ (\phi_{M}\ot \lambda_H)\circ (M\ot \delta_{H})=\phi_{M}\co (M\ot \Pi_{H}^{L}).$
\item[(h2-5)] $\phi_{M}\circ ((\phi_{M}\co (M\ot \lambda_{H}))\ot H)\circ (M\ot \delta_{H})=\phi_{M}\co (M\ot \Pi_{H}^{R}),$
\end{itemize}
\end{itemize}
because  in this case $B^{co H}=H_{L}$ and $i_{B}=i_{L}$.

\underline{Examples for Hopf quasigroups}: The following example comes  from the nonassociative setting. Let $H$ be a Hopf quasigroup and assume that $B=H$, $\rho_{B}=\delta_{H}$, $h=id_{H}$. In this case  the equalities (\ref{strong-1})  and (\ref{strong2}) hold  because $q_{H}=\eta_{H}\ot \varepsilon_{H}$, $H^{coH}=K$, and $i_{H}=\eta_{H}$. Then the category ${\mathcal S}{\mathcal M}^{H}_{H}(id_{H})$ is the one whose objects are triples $(M, \phi_{M}, \rho_{M})$ where:
\begin{itemize}
\item[(i1)] The pair $(M, \rho_{M})$ is a right $H$-comodule.
\item[(i2)] The morphism $\phi_{M}:M\ot H\rightarrow M$ satisfies:
\begin{itemize}
\item[(i2-1)] $\phi_{M}\co (M\ot \eta_{H})=id_{M},$

\item[(i2-2)] $\rho_{M}\co \phi_{M}=(\phi_{M}\ot \mu_{H})\co (M\ot c_{H,H}\ot H)\co (\rho_{M}\ot \delta_{H}),$
\item[(i2-3)] $\phi_{M}\circ (\phi_{M}\ot (M\co\lambda_H))\circ (M\ot \delta_{H})=M\ot \varepsilon_{H},$
\item[(i2-4)] $\phi_{M}\circ (\phi_{M}\ot \lambda_{H})\circ (M\ot \delta_{H})=M\ot \varepsilon_{H}.$
\end{itemize}
\end{itemize}

Note that in this setting the equality $\phi_{M}\co ((\phi_{M}\co (M\ot i_{H}))\ot H)=\phi_{M}\co (M\ot (\mu_{H}\co (i_{H}\ot H)))$ holds because $i_{H}=\eta_{H}$. The morphisms of ${\mathcal S}{\mathcal M}^{H}_{H}(id_{H})$ are morphisms of right $H$-comodules and  $H$-quasilinear, i.e., satisfying (\ref{quasilineal}), where $\omega_{M}=\phi_{M}\co (i_{M}\ot H):M^{co H}\ot H\rightarrow M$ is the associated isomorphism with inverse $\omega_{M}^{-1}=(p_{M}\ot H)\co \rho_{M}$, and $\phi_{M^{coH}\ot H}=M^{coH}\ot \mu_{H}$ and $\rho_{M^{coH}\ot H}=M^{co H}\ot \delta_{H}$. Therefore ${\mathcal S}{\mathcal M}^{H}_{H}(id_{H})$ is the category of Hopf modules introduced by Brzezi\'nski in \cite{Brz}. 

In the previous nonassociative setting,  let $(B,\rho_{B})$ be a right $H$-comodule magma such that the functors $-\ot B$ and $-\ot H$ preserve coequalizers, and such that the equalities (\ref{strong-1})  and (\ref{strong2}) hold. For any  anchor morphism $h$,  the category ${\mathcal S}{\mathcal M}^{H}_{B}(h)$ is the one whose objects are triples $(M, \phi_{M}, \rho_{M})$ where:
\begin{itemize}
\item[(j1)] The pair $(M, \rho_{M})$ is a right $H$-comodule.
\item[(j2)] The morphism $\phi_{M}:M\ot B\rightarrow M$ satisfies:
\begin{itemize}
\item[(j2-1)] $\phi_{M}\co (M\ot \eta_{B})=id_{M},$
\item[(j2-2)] $\phi_{M}\co ((\phi_{M}\co (M\ot i_{B}))\ot B)=\phi_{M}\co (M\ot (\mu_{B}\co (i_{B}\ot B))),$
\item[(j2-3)] $\rho_{M}\co \phi_{M}=(\phi_{M}\ot \mu_{H})\co (M\ot c_{H,B}\ot H)\co (\rho_{M}\ot \rho_{B})$,
\item[(j2-4)] $\phi_{M}\circ ((\phi_{M}\co (M\ot h))\ot (h\co \lambda_{H}))\co (M\ot \delta_{H})=M\ot \varepsilon_{H},$
\item[(j2-5)] $\phi_{M}\circ ((\phi_{M}\co (M\ot (h\co \lambda_{H})))\ot h)\circ (M\ot \delta_{H})=M\ot \varepsilon_{H},$
\end{itemize}
\end{itemize}
because in this case $\Pi_{H}^{L}=\Pi_{H}^{R}=\eta_{H}\ot \varepsilon_{H}$. The morphisms of ${\mathcal S}{\mathcal M}^{H}_{B}(h)$ are morphisms of right $H$-comodules and  $B$-quasilinear where $\omega_{M}:M^{co H}\ot_{B^{coH}} H\rightarrow M$ is the associated isomorphism of right $H$-comodules defined in the proof of Proposition \ref{tensor}. 

\underline{Examples for weak Hopf quasigroups}: If $H$ is a weak Hopf quasigroup ${\mathcal S}{\mathcal M}^{H}_{H}(id_{H})$  is the category of strong  Hopf modules defined in \cite{Strong} and denoted by ${\mathcal S}{\mathcal M}^{H}_{H}$. Note that, as a consequence of (\ref{d2-7}), we obtain that 
$$\phi_{M}\circ (\phi_{M} \ot H)\co (M\ot \Pi_{H}^{L}\ot H)\circ (M\ot \delta_{H})=\phi_{M}$$
is a superfluous identity in the definition of strong  Hopf module introduced in \cite{Strong}. 
\end{example}

\begin{proposition}
\label{struc-tim-2}
Assume that  the conditions of Proposition \ref{tensor} hold. Let  $(M, \phi_{M}, \rho_{M})$ be an object in ${\mathcal S}{\mathcal M}^{H}_{B}(h)$. Let $\omega_{M}$ be the isomorphism of right $H$-comodules between $M^{coH}\ot_{B^{coH}}B$ and $M$. Then the triple $(M, \phi_{M}^{\omega_{M}}, \rho_{M})$  is a strong $(H,B,h)$-Hopf module  with the same object of coinvariants that  $(M, \phi_{M}, \rho_{M})$. Also, the identity 
\begin{equation}
\label{action-w}
\phi_{M}^{\omega_{M}}=\phi_{M}\co (q_{M}\ot (\mu_{B}\co (h\ot B)))\co (\rho_{M}\ot B)
\end{equation}
holds and 
\begin{equation}
\label{q-w}
q_{M}^{\omega_{M}}=q_{M},
\end{equation} 
where $q_{M}^{\omega_{M}}=\phi_{M}^{\omega_{M}}\co (M\ot (h\co \lambda_{H}))\co \rho_{M}$ is the idempotent morphism associated to the 
Hopf module $(M, \phi_{M}^{\omega_{M}}, \rho_{M})$. Moreover, for  $(M, \phi_{M}^{\omega_{M}}, \rho_{M})$, the associated  isomorphism of right $H$-comodules between $M^{coH}\ot_{B^{coH}}B$ and $M$ is $\omega_{M}$, and the equality  
\begin{equation}
\label{phi-square}
(\phi_{M}^{\omega_{M}})^{\omega_{M}}=
\phi_{M}^{\omega_{M}}
\end{equation}
holds. Finally,  there exists an idempotent functor  
$$D:{\mathcal S}{\mathcal M}^{H}_{B}(h)\rightarrow  {\mathcal S}{\mathcal M}^{H}_{B}(h),$$ called the  $\omega$-deformation functor, defined on objects by $D((M,\phi_{M},\rho_{M}))=(M, \phi_{M}^{\omega_{M}}, \rho_{M})$ and on morphisms by the identity.
\end{proposition}

\begin{proof} By Proposition \ref{isomorp}, $(M, \phi_{M}^{\omega_{M}}, \rho_{M})$ is a strong $(H,B,h)$-Hopf module. Moreover, note that (\ref{action-w}) follows by (\ref{phi-ten}) and (\ref{omeg}). Then $q_{M}^{\omega_{M}}=q_{M}$ because 

\begin{itemize}
\item[ ]$\hspace{0.38cm} q_{M}^{\omega_{M}}$

\item[ ]$=\phi_{M}\co (q_{M}\ot (\mu_{B}\co (h\ot h)))\co (\rho_{M}\ot \lambda_{H})\co \rho_{M}  $ {\scriptsize  ({\blue (\ref{action-w})})}

\item[ ]$= \phi_{M}\co (q_{M}\ot (h\co \Pi_{H}^{L}))\co \rho_{M} $ {\scriptsize  ({\blue comodule condition for $M$ and multiplicative condition for $h$})}

\item[ ]$= \phi_{M}\circ (\phi_{M}\ot B)\co (q_{M}\ot h\ot (h\co \lambda_{H}))\co (M\ot \delta_{H})\co \rho_{M} $ {\scriptsize  ({\blue (d2-4) of Definition \ref{H-D-mod}})} 

\item[ ]$= \phi_{M}\co ((\phi_{M}\co (q_{M}\ot h)\co \rho_{M})\ot (h\co \lambda_{H}))\co \rho_{M} $ {\scriptsize  ({\blue comodule condition for $M$})}

\item[ ]$= q_{M} ${\scriptsize  ({\blue (\ref{l-idem4})})}. 

\end{itemize}

Therefore, $(M, \phi_{M}^{\omega_{M}}, \rho_{M})$  has the same object of coinvariants that $(M, \phi_{M}, \rho_{M})$. On the other hand, 

\begin{itemize}
\item[ ]$\hspace{0.38cm} (\phi_{M}^{\omega_{M}})^{\omega_{M}} $

\item[ ]$=  \phi_{M}\co (q_{M}\ot \mu_{B})\co (((M\ot h)\co \rho_{M}\co q_{M})\ot (\mu_{B}\co (h\ot B)))\co (\rho_{M}\ot B) $ {\scriptsize  ({\blue (\ref{action-w})})}

\item[ ]$=  \phi_{M}\co (q_{M}\ot \mu_{B})\co (((M\ot (h\co \Pi_{H}^{L}))\co \rho_{M}\co q_{M})\ot (\mu_{B}\co (h\ot B)))\co (\rho_{M}\ot B) $ {\scriptsize  ({\blue (\ref{idemp-0M})})}

\item[ ]$=\phi_{M}\co (q_{M}\ot \mu_{B})\co (((M\ot (q_{B}\co h))\co \rho_{M}\co q_{M})\ot (\mu_{B}\co (h\ot B)))\co (\rho_{M}\ot B)    $ {\scriptsize  ({\blue (\ref{anchor1})})}

\item[ ]$= \phi_{M}\co ((\phi_{M}\co (q_{M}\ot (q_{B}\co h))\co \rho_{M}\co q_{M})\ot (\mu_{B}\co (h\ot B)))\co (\rho_{M}\ot B)   $ {\scriptsize  ({\blue (d2-2) of Definition \ref{H-D-mod}})} 

\item[ ]$= \phi_{M}\co ((\phi_{M}\co (q_{M}\ot (h\co \Pi_{H}^{L}))\co \rho_{M}\co q_{M})\ot (\mu_{B}\co (h\ot B)))\co (\rho_{M}\ot B)   $ {\scriptsize  ({\blue (\ref{anchor1})})}

\item[ ]$=\phi_{M}\co ((\phi_{M}\co (q_{M}\ot h)\co \rho_{M}\co q_{M})\ot (\mu_{B}\co (h\ot B)))\co (\rho_{M}\ot B)    $ {\scriptsize  ({\blue (\ref{idemp-0M})})}

\item[ ]$=\phi_{M}^{\omega_{M}}$ {\scriptsize  ({\blue (\ref{l-idem4})})}. 

\end{itemize}

Finally, by (\ref{q-w}) and (\ref{phi-square}), it is easy to show that  $D$  is a well defined idempotent endofunctor.
\end{proof}

\begin{lemma}
Assume that  the conditions of Proposition \ref{tensor} hold. Let  $(M, \phi_{M}, \rho_{M})$ be an object in ${\mathcal S}{\mathcal M}^{H}_{B}(h)$. The following identity holds:
\begin{equation}
\label{iphi}
\phi_{M}^{\omega_{M}}\co (i_{M}\ot B)=\phi_{M}\co (i_{M}\ot B).
\end{equation}
\end{lemma}

\begin{proof} Indeed, 

\begin{itemize}
\item[ ]$\hspace{0.38cm} \phi_{M}^{\omega_{M}}\co (i_{M}\ot B) $

\item[ ]$= \phi_{M}\co (q_{M}\ot \mu_{B})\co (((M\ot (h\co \Pi_{H}^{L}))\co\rho_{M}\co i_{M})\ot B)  $ {\scriptsize  ({\blue (\ref{idemp-0M})})}

\item[ ]$= \phi_{M}\co (q_{M}\ot \mu_{B})\co (((M\ot (q_{B}\co h))\co\rho_{M}\co i_{M})\ot B)   $ {\scriptsize  ({\blue (\ref{anchor1})})}

\item[ ]$= \phi_{M}\co ((\phi_{M}\co (q_{M}\ot (q_{B}\co h))\co\rho_{M}\co i_{M})\ot B)   $ {\scriptsize  ({\blue (d2-2) of Definition \ref{H-D-mod}})}

\item[ ]$=\phi_{M}\co ((\phi_{M}\co (q_{M}\ot (h\co \Pi_{H}^{L}))\co\rho_{M}\co i_{M})\ot B)    $ {\scriptsize  ({\blue (\ref{anchor1})})} 

\item[ ]$= \phi_{M}\co ((\phi_{M}\co (q_{M}\ot h)\co\rho_{M}\co i_{M})\ot B)  $ {\scriptsize  ({\blue (\ref{idemp-0M})})}

\item[ ]$=\phi_{M}\co (i_{M}\ot B)$ {\scriptsize  ({\blue (\ref{l-idem4})})}. 

\end{itemize}

\end{proof}

\begin{proposition}
\label{inv-defor}
Assume that  the conditions of Proposition \ref{tensor} hold. For any object  $(M, \phi_{M}, \rho_{M})$  in ${\mathcal S}{\mathcal M}^{H}_{B}(h)$, the  strong $(H,B,h)$-Hopf module $(M^{coH}\ot_{B^{coH}}B, \phi_{M^{coH}\ot_{B^{coH}}B}, \rho_{M^{coH}\ot_{B^{coH}}B})$, constructed in Proposition \ref{tensor}, is invariant for the  $\omega$-deformation functor, i.e., 
$$D((M^{coH}\ot_{B^{coH}}B, \phi_{M^{coH}\ot_{B^{coH}}B}, \rho_{M^{coH}\ot_{B^{coH}}B}))=(M^{coH}\ot_{B^{coH}}B, \phi_{M^{coH}\ot_{B^{coH}}B}, \rho_{M^{coH}\ot_{B^{coH}}B}).$$
\end{proposition}

\begin{proof} To prove the proposition we only need to show that 
\begin{equation}
\label{w-MB-inv}
\phi_{M^{coH}\ot_{B^{coH}}B}^{\omega_{M^{coH}\ot_{B^{coH}}B}}=\phi_{M^{coH}\ot_{B^{coH}}B}. 
\end{equation}

Indeed, composing with the coequalizer $n_{M^{coH}}\ot B$ we have 
$$\phi_{M^{coH}\ot_{B^{coH}}B}^{\omega_{M^{coH}\ot_{B^{coH}}B}}\co (n_{M^{coH}}\ot B)\stackrel{{\scriptsize \blue  (\ref{rho-ten}), (\ref{phi-ten})}}{=} n_{M^{coH}}\co (M^{coH}\ot (\mu_{B}\co (q_{B}\ot (\mu_{B}\co (h\ot B)))\co (\rho_{B}\ot B)))$$
$$\stackrel{{\scriptsize \blue  (\ref{strong2})}}{=} n_{M^{coH}}\co (M^{coH}\ot 
(\mu_{B}\co ((\mu_{B}\co (q_{B}\ot h)\co \rho_{B})\ot B)))\stackrel{{\scriptsize \blue  (\ref{anchor21})}}{=} n_{M^{coH}}\co (M^{coH}\ot \mu_{B})\stackrel{{\scriptsize \blue  (\ref{phi-ten})}}{=}\phi_{M^{coH}\ot_{B^{coH}}B}\co (n_{M^{coH}}\ot B).$$
Therefore, (\ref{w-MB-inv}) holds.
\end{proof}

The following result is the nonassociative general version of the Fundamental Theorem of Hopf Modules. The proof follows by  the properties of the morphism $\omega_{M}$, obtained in Proposition \ref{tensor}, and by (\ref{w-MB-inv}).

\begin{theorem} ({\rm Fundamental Theorem of Hopf modules})
\label{main0}
Assume that  the conditions of Proposition \ref{tensor} hold. Let  $(M, \phi_{M}, \rho_{M})$ be an object in  ${\mathcal S}{\mathcal M}^{H}_{B}(h)$.  The objects $(M^{coH}\ot_{B^{coH}}B, \phi_{M^{coH}\ot_{B^{coH}}B}, \rho_{M^{coH}\ot_{B^{coH}}B})$ and $(M, \phi_{M}, \rho_{M})$ are isomorphic in ${\mathcal S}{\mathcal M}^{H}_{B}(h)$.
\end{theorem}

The previous theorem is a generalization of the one proved by Sweedler \cite{Sweedler} for Hopf modules over an ordinary Hopf algebra. It also contains the Fundamental Theorem of relative Hopf modules (or $(H,B)$-Hopf modules, or Doi-Hopf modules) given by Doi and Takeuchi in \cite{Doi-Take}. On the other hand, in the weak setting, Theorem \ref{main0} is a generalization of the one obtained by B\"{o}hm, Nill and  Szlach\'anyi \cite{bohm}, for Hopf modules over a weak Hopf algebra $H$,  and the one proved by Zhang and Zhu \cite{ZZ}, for $(H,B)$-Hopf modules associated to a weak right $H$-comodule algebra $B$. Moreover, in the nonassociative context,  it generalizes the result obtained by Brzezi\'nski \cite{Brz} for Hopf modules associated to a  Hopf quasigroup. Finally, for weak Hopf quasigroups, Theorem \ref{main0} is a generalization of the Fundamental Theorem of Hopf modules proved in \cite{Asian} (see also \cite{Strong}).

\section{Categorical equivalences for strong $(H,B,h)$-Hopf modules}

As for prerequisites, in this section we will assume that  the conditions of Proposition \ref{tensor} hold. Then, in the following $H$ is a weak Hopf quasigroup and  $(B,\rho_{B})$ is a right $H$-comodule
magma. Also,  $h:H\rightarrow B$ is an anchor morphism such that (\ref{strong-1})  and (\ref{strong2})
hold. Finally, the category ${\mathcal C}$ admits coequalizers and the functors $-\ot B$ and $-\ot H$ preserve coequalizers. With ${\mathcal C}_{B^{coH}}$ we will denote the category of right $B^{coH}$-modules.

The main target of this section is to prove that there exists an equivalence between ${\mathcal C}_{B^{coH}}$  and the category of strong $(H,B,h)$-Hopf modules. 

 Let $(N,\psi_{N})$ be an object in ${\mathcal C}_{B^{coH}}$ and consider the coequalizer diagram 

\begin{equation}
\label{coeqN-1}
\setlength{\unitlength}{1mm}
\begin{picture}(101.00,10.00)
\put(20.00,8.00){\vector(1,0){25.00}}
\put(20.00,4.00){\vector(1,0){25.00}}
\put(62.00,6.00){\vector(1,0){21.00}}
\put(32.00,11.00){\makebox(0,0)[cc]{$\psi_{N}\ot B$ }}
\put(33.00,0.5){\makebox(0,0)[cc]{$N\ot (\mu_{B}\co (i_{B}\ot B))$}} 
\put(73.00,9.00){\makebox(0,0)[cc]{$n_{N}$ }}
\put(7.00,6.00){\makebox(0,0)[cc]{$ N\otimes B^{coH}\ot B$ }}
\put(54.00,6.00){\makebox(0,0)[cc]{$ N\ot B$ }}
\put(96.00,6.00){\makebox(0,0)[cc]{$N\ot_{B^{coH}} B. $ }}
\end{picture}
\end{equation}

Then, 
$$(n_{N}\ot B)\co (\psi_{N}\ot \rho_{B})\stackrel{{\scriptsize \blue  (\ref{coeqN-1})}}{=}(n_{N}\ot B)\co (N\ot ((\mu_{B}\ot B)\co (i_{B}\ot \rho_{B})))\stackrel{{\scriptsize \blue  (\ref{l-idem0})}}{=}
(n_{N}\ot B)\co (N\ot (\rho_{B}\co (\mu_{B}\co (i_{B}\ot B))))$$ 
and, as a consequence, there exists a unique morphism $\rho_{N\ot_{B^{coH}} B}:N\ot_{B^{coH}} B\rightarrow N\ot_{B^{coH}} B\ot H$ such that 
\begin{equation}
\label{comodN}
\rho_{N\ot_{B^{coH}} B}\co n_{N}=(n_{N}\ot B)\co (N\ot \rho_{B}).
\end{equation}

On the other hand, by (\ref{strong2}), we have 
$$n_{N}\co (\psi_{N}\ot \mu_{B})=n_{N}\co (N\ot (\mu_{B}\co (i_{B}\ot \mu_{B})))=n_{N}\co (N\ot (\mu_{B}\co ((\mu_{B}\co (i_{B}\ot B))\ot B))),$$
and then,using that the functor $-\ot B$ preserves coequalizers, there exists a unique morphism 
$$\phi_{N\ot_{B^{coH}} B}:N\ot_{B^{coH}} B\ot B\rightarrow N\ot_{B^{coH}} B$$ such that 
\begin{equation}
\label{quasi-modN}
\phi_{N\ot_{B^{coH}} B}\co (n_{N}\ot B)=n_{N}\co (N\ot \mu_{B}).
\end{equation}

By a similar proof to the one used for $M^{coH}\ot_{B^{coH}} B$ in Proposition \ref{tensor},  we can prove that $$(N\ot_{B^{coH}} B, \phi_{N\ot_{B^{coH}} B}, \rho_{N\ot_{B^{coH}} B})$$ is a strong $(H,B,h)$-Hopf module such that (the proof follows the ideas given in Proposition \ref{inv-defor} for $M^{coH}\ot_{B^{coH}} B$)
\begin{equation}
\label{w-d2}
\phi_{N\ot_{B^{coH}} B}^{\omega_{N\ot_{B^{coH}} B}}=\phi_{N\ot_{B^{coH}} B}.
\end{equation}

On the other hand, if $f:N\rightarrow P$ is a morphism in ${\mathcal C}_{B^{coH}}$, we have 
$$n_{P}\co (f\ot B)\co (\psi_{N}\ot B)=n_{P}\co (f\ot B)\co (N\ot (\mu_{B}\co (i_{B}\ot B)))$$
and then there exists a unique morphism $f\ot_{B^{coH}}B:N\ot_{B^{coH}} B\rightarrow P\ot_{B^{coH}} B$ such that 
\begin{equation}
\label{mor-induc}
n_{P}\co (f\ot B)=(f\ot_{B^{coH}}B)\co n_{N}.
\end{equation}

The morphism $f\ot_{B^{coH}}B$ is a morphism in ${\mathcal S}{\mathcal M}^{H}_{B}(h)$ because 
$$\rho_{P\ot_{B^{coH}} B}\co f\ot_{B^{coH}} B \co n_{N}\stackrel{{\scriptsize \blue  (\ref{comodN}), (\ref{mor-induc})}}{=}(n_{P}\ot B)\co (f\ot \rho_{B})\stackrel{{\scriptsize \blue  (\ref{comodN})}}{=}\rho_{P\ot_{B^{coH}} B}\co n_{P}\co (f\ot B)$$
$$\stackrel{{\scriptsize \blue  (\ref{mor-induc})}}{=}(f\ot_{B^{coH}} B\ot B)\co 
\rho_{N\ot_{B^{coH}} B}\co n_{N}$$
and 
$$  
\phi_{P\ot_{B^{coH}} B}^{\omega_{P\ot_{B^{coH}} B}}\co (f\ot_{B^{coH}} B\ot B)\co (n_{N}\ot B)
\stackrel{{\scriptsize \blue  (\ref{w-d2})}}{=}
\phi_{P\ot_{B^{coH}} B}\co (f\ot_{B^{coH}} B\ot B)\co (n_{N}\ot B)
\stackrel{{\scriptsize \blue  (\ref{quasi-modN}), (\ref{mor-induc})}}{=}
n_{P}\co (f\ot \mu_{B})$$
$$
\stackrel{{\scriptsize \blue  (\ref{mor-induc})}}{=}
f\ot_{B^{coH}} B\co n_{N}\co (N\ot \mu_{B})
\stackrel{{\scriptsize \blue  (\ref{quasi-modN})}}{=}
f\ot_{B^{coH}} B\co\phi_{N\ot_{B^{coH}} B}\co (n_{N}\ot B)
\stackrel{{\scriptsize \blue  (\ref{w-d2})}}{=}
f\ot_{B^{coH}} B\co\phi_{N\ot_{B^{coH}} B}^{\omega_{N\ot_{B^{coH}} B}}\co (n_{N}\ot B).$$

Summarizing, we have the following proposition:

\begin{proposition}
\label{I-functor}
There exists a functor $F:{\mathcal C}_{B^{coH}}\rightarrow  {\mathcal S}{\mathcal M}^{H}_{B}(h)$, called the induction functor, defined on objects by $F((N,\psi_{N}))=(N\ot_{B^{coH}} B, \phi_{N\ot_{B^{coH}} B}, \rho_{N\ot_{B^{coH}} B})$ and on morphisms by $F(f)=f\ot_{B^{coH}}B$.
\end{proposition}

Now let $(M,\phi_{M}, \rho_{M})$ be an object in $ {\mathcal S}{\mathcal M}^{H}_{B}(h)$. By Proposition 
\ref{M-coinv} we have that  the object of coinvariants $M^{coH}$ is a right $B^{coH}$-module where 
$\phi_{M^{coH} }=p_{M}\co \phi_{M}\co (i_{M}\ot i_{B})$ (see (\ref{PhiBcoH1})). Let $g:M\rightarrow Q$ be a morphism in  $ {\mathcal S}{\mathcal M}^{H}_{B}(h)$. Using the comodule morphism condition we obtain that $\rho_{Q}\circ g\circ i_{M}=(Q\ot \overline{\Pi}_{H}^{R})\co \rho_{Q}\circ g\circ i_{M}$ and this implies that there exists a unique morphism $g^{coH}:M^{coH}\rightarrow Q^{coH}$ such that 
\begin{equation}
\label{coi-morph}
i_{Q}\co g^{coH}=g\co i_{M}.
\end{equation}
Also, using that $g$ is $B$-quasilinear, $H$ colinear, and (\ref{q-w}) we have 
\begin{equation}
\label{Bq-Hc}
g\co q_{M}^{\omega_{M}}=g\co q_{M}=q_{T}\co g=q_{T}^{\alpha_{T}}\co g.
\end{equation}

Then, 
$$i_{Q}\co g^{coH}\co p_{M}
\stackrel{{\scriptsize \blue  (\ref{coi-morph})}}{=}
g\co q_{M}
\stackrel{{\scriptsize \blue  (\ref{Bq-Hc})}}{=}
q_{Q}\co g$$
and, as a consequence,
\begin{equation}
\label{coinv-m-1}
g^{coH}\co p_{M}=p_{Q}\co g
\end{equation}
holds. 

The morphism $g^{coH}$ is a morphism in ${\mathcal C}_{B^{coH}}$ because

\begin{itemize}
\item[ ]$\hspace{0.38cm} \phi_{Q^{coH}}\co (g^{coH}\ot B^{coH}) $

\item[ ]$= p_{Q}\co \phi_{Q}\co ((i_{Q}\co  g^{coH})\ot i_{B}) $ {\scriptsize  ({\blue (\ref{PhiBcoH1})})}

\item[ ]$=p_{Q}\co \phi_{Q}^{\omega_{Q}}\co ((i_{Q}\co  g^{coH})\ot i_{B})    $ {\scriptsize  ({\blue (\ref{iphi})})}

\item[ ]$=  p_{Q}\co \phi_{Q}^{\omega_{Q}}\co ((g\co  i_{M})\ot i_{B})   $ {\scriptsize  ({\blue (\ref{coi-morph})})}

\item[ ]$=  p_{Q}\co g\co \phi_{M}^{\omega_{M}}\co ( i_{M}\ot i_{B})  $ {\scriptsize  ({\blue (\ref{quasilineal})})} 

\item[ ]$= p_{Q}\co g\co \phi_{M}\co ( i_{M}\ot i_{B})  $ {\scriptsize  ({\blue (\ref{iphi})})}

\item[ ]$=g^{coH}\co \phi_{M^{coH}}$ {\scriptsize  ({\blue (\ref{coinv-m-1})})}. 

\end{itemize}

Thus, we have the following result.

\begin{proposition}
\label{c-functor}
There exists a functor $G:{\mathcal S}{\mathcal M}^{H}_{B}(h)\rightarrow  {\mathcal C}_{B^{coH}},$ called the  functor of coinvariants, defined on objects by $G((M,\phi_{M},\rho_{M}))=(M^{coH}, \psi_{M^{coH}})$ and on morphisms by $G(g)=g^{coH}$.
\end{proposition}

\begin{theorem}
\label{p-paper}
The categories ${\mathcal SM}^{H}_{B}(h)$ and ${\mathcal C}_{B^{coH}}$ are equivalent.
\end{theorem}

\begin{proof} To prove the theorem,   we  firstly obtain that the induction functor $F$, introduced in Proposition \ref{I-functor}, is left adjoint to the functor of coinvariants $G$ introduced in Proposition \ref{c-functor}. Later, we  show that  the unit and counit  associated to this adjunction are natural isomorphisms. Then, we  proceed as in the proof of Theorem 3.10 of \cite{Strong} by dividing the proof in three steps. 

{\it \underline{Step 1:}} In this step we  define the unit of the adjunction. For any right $B^{coH}$-module  $(N, \psi_{N})$, consider 
$$\alpha_{N}:N\rightarrow GF(N)=(N\ot_{B^{coH}}B)^{coH}$$
as the unique morphism such that 
\begin{equation}
\label{unit}
i_{N\ot_{B^{coH}}B}\co \alpha_{N}=n_{N}\co (N\ot \eta_{B}). 
\end{equation}

This morphism exists and is unique because 
$$((N\ot_{B^{coH}}B)\ot \overline{\Pi}_{H}^{R})\co \rho_{N\ot_{B^{coH}}B}\co n_{N}\co (N\ot \eta_{B})
\stackrel{{\scriptsize \blue  (\ref{comodN})}}{=}
(n_{N}\ot H)\co (N\ot ((B\ot \overline{\Pi}_{H}^{R})\co\rho_{B}\co \eta_{B}))
$$
$$
\stackrel{{\scriptsize \blue  ({\rm b5})}}{=}
(n_{N}\ot H)\co (N\ot (\rho_{B}\co \eta_{B}))
\stackrel{{\scriptsize \blue  (\ref{comodN})}}{=}
\rho_{N\ot_{B^{coH}}B}\co n_{N}\co (N\ot \eta_{B}).$$

Also, $\alpha_{N}$ is a morphism in  ${\mathcal C}_{B^{coH}}$. Indeed: Composing with the equalizer $i_{N\ot_{B^{coH}}B}$ we have 

\begin{itemize}
\item[ ]$\hspace{0.38cm} i_{N\ot_{B^{coH}}B} \co \psi_{(N\ot_{B^{coH}}B)^{coH}}\co (\alpha_{N}\ot B^{coH}) $
\item[ ]$= q_{N\ot_{B^{coH}}B} \co \phi_{N\ot_{B^{coH}}B}\co ((i_{N\ot_{B^{coH}}B}\co \alpha_{N})\ot i_{B})$ {\scriptsize ({\blue (\ref{PhiBcoH1})})}
\item[ ]$= q_{N\ot_{B^{coH}}B} \co \phi_{N\ot_{B^{coH}}B}\co ((n_{N}\co (N\ot \eta_{B}))\ot i_{B})$ {\scriptsize ({\blue (\ref{unit})})}
\item[ ]$= q_{N\ot_{B^{coH}}B} \co n_{N}\co (N\ot (\mu_{B}\co (\eta_{B}\ot i_{B})))$ {\scriptsize ({\blue (\ref{quasi-modN})})}
\item[ ]$= n_{N}\co (N\ot (q_{B}\co  i_{B}))$ {\scriptsize ({\blue properties of $\eta_{B}$, (\ref{comodN}), and (\ref{quasi-modN})})}
\item[ ]$=n_{N}\co (N\ot  i_{B}) $ {\scriptsize ({\blue properties of $i_{B}$})}
\item[ ]$=n_{N}\co (N\ot (\mu_{B}\co (i_{B}\ot \eta_{B}))) $ {\scriptsize ({\blue properties of $\eta_{B}$})}
\item[ ]$=n_{N}\co (\psi_{N}\ot \eta_{B}) $ {\scriptsize ({\blue (\ref{coeqN-1})})}
\item[ ]$=i_{N\ot_{B^{coH}}B} \co \alpha_{N}\co \psi_{N} $ {\scriptsize ({\blue (\ref{unit})})}
\end{itemize}
and therefore $\psi_{N\ot_{B^{coH}}B}\co (\alpha_{N}\ot B^{coH})=\alpha_{N}\co \psi_{N}.$ On the other hand, the morphism $\alpha_{N}$ is natural in $N$ because if $f:N\rightarrow P$ is a morphism in ${\mathcal C}_{B^{coH}}$  
$$ i_{P\ot_{B^{coH}}B}\co (f\ot_{B^{coH}}B)^{coH}\co \alpha_{N}
\stackrel{{\scriptsize \blue  (\ref{coi-morph})}}{=}
(f\ot_{B^{coH}}B)\co i_{N\ot_{B^{coH}}B}\co \alpha_{N}
\stackrel{{\scriptsize \blue  (\ref{unit})}}{=}
(f\ot_{B^{coH}}B)\co n_{N}\co (N\ot \eta_{B})
$$
$$
\stackrel{{\scriptsize \blue  (\ref{mor-induc})}}{=}
n_{P}\co (f\ot \eta_{B})
\stackrel{{\scriptsize \blue  (\ref{unit})}}{=}
i_{P\ot_{B^{coH}}B}\co \alpha_{P}\co f,$$
and then $(f\ot_{B^{coH}}B)^{coH}\co \alpha_{N}=\alpha_{P}\co f$.

Finally, we  prove that $\alpha_{N}$ is an isomorphism for all right $B^{coH}$-module $N$. First note that, under the conditions of this theorem, the triple $(B, \mu_{B}, \delta_{B})$ is a  strong $(H,B, h)$-Hopf module, and then 
$$\psi_{N}\co (\psi_{N}\ot p_{B})=\psi_{N}\co
 (N\ot (p_{B}\co \mu_{B}\co (i_{B}\ot B)))$$
because 
\begin{itemize}
\item[ ]$\hspace{0.38cm}  \psi_{N}\co (\psi_{N}\ot p_{B})$
\item[ ]$=\psi_{N}\co (N\ot (\mu_{B^{coH}}\co (B^{coH}\ot p_{B})))  $ {\scriptsize ({\blue module condition for $N$})}
\item[ ]$=\psi_{N}\co (N\ot (p_{B}\co \mu_{B}\co (i_{B}\ot q_{B}))  $ {\scriptsize ({\blue (\ref{muH1})})}
\item[ ]$= \psi_{N}\co
 (N\ot (p_{B}\co \mu_{B}\co (i_{B}\ot B))) $ {\scriptsize ({\blue (\ref{l-idem3})})}.
\end{itemize}

Therefore, there exists a unique morphism $m_{N}:N\ot_{B^{coH}}B\rightarrow N$ such that 
\begin{equation}
\label{mn}
m_{N}\co n_{N}=\psi_{N}\co (N\ot p_{B}). 
\end{equation}

Now define the morphism $x_{N}:(N\ot_{B^{coH}}B)^{coH}\rightarrow N$ by $x_{N}=m_{N}\co i_{N\ot_{B^{coH}}B}$. Then, 
$$x_{N}\co \alpha_{N}
\stackrel{{\scriptsize \blue  (\ref{unit})}}{=}
m_{N}\co n_{N}\co (N\ot \eta_{B})
\stackrel{{\scriptsize \blue  (\ref{mn})}}{=}
\psi_{N}\co (N\ot (p_{B}\co \eta_{B}))
\stackrel{{\scriptsize \blue  (\ref{etaH})}}{=}
\psi_{N}\co (N\ot  \eta_{B^{coH}})=
id_{N}.
$$

On the other hand, composing with 
$i_{N\ot_{B^{coH}}B}$ and $p_{N\ot_{B^{coH}}B}\co n_{N}$  we have 

\begin{itemize}
\item[ ]$\hspace{0.38cm}  i_{N\ot_{B^{coH}}B} \co  \alpha_{N}\co x_{N}\co p_{N\ot_{B^{coH}}N}\co n_{N}$
\item[ ]$= n_{N}\co ((m_{N}\co q_{N\ot_{B^{coH}}B}\co n_{N})\ot \eta_{B}) $ {\scriptsize ({\blue (\ref{unit})})}
\item[ ]$= n_{N}\co ((\psi_{N}\co (N\ot (p_{B}\co q_{B})))\ot \eta_{B}) $ {\scriptsize ({\blue (\ref{comodN}), (\ref{quasi-modN}), (\ref{mn})})}
\item[ ]$=n_{N}\co ((\psi_{N}\co (N\ot p_{B}))\ot \eta_{B})  $ {\scriptsize ({\blue properties of $q_{B}$})}
\item[ ]$= n_{N}\co (N\ot (\mu_{B}\co (q_{B}\ot \eta_{B}))) $ {\scriptsize ({\blue (\ref{coeqN-1})})}
\item[ ]$= n_{N}\co (N\ot q_{B}) $ {\scriptsize ({\blue properties of $\eta_{B}$})}
\item[ ]$= q_{N\ot_{B^{coH}}B} \co n_{N} $ {\scriptsize ({\blue (\ref{comodN}), (\ref{quasi-modN})})}.
\end{itemize}

Therefore, $\alpha_{N}\co x_{N}=id_{(N\ot_{B^{coH}}B)^{coH}}$ and, as a consequence, $\alpha_{N}$ is an isomorphism.

{\it \underline{Step 2:}} For any $(M,\phi_{M},\rho_{M})\in {\mathcal SM}^{H}_{B}(h)$ the counit is defined as 
$ \beta_{M}=\omega_{M}:M^{coH}\ot_{B^{coH}}B\rightarrow M,$
where $\omega_{M}$ is the isomorphism satisfying (\ref{omeg}). By Theorem \ref{main0}, we know that $\omega_{M}$ is an isomorphism in ${\mathcal SM}^{H}_{B}(h)$, and it is  natural because if $g:M\rightarrow Q$ is a morphism in ${\mathcal SM}^{H}_{B}(h)$ we have 
$$\beta_{Q}\co (g^{coH}\ot_{B^{coH}}	\ot B)\co n_{M^{coH}}
\stackrel{{\scriptsize \blue  (\ref{mor-induc})}}{=}
\beta_{Q}\co n_{Q^{coH}}\co (g^{coH}\ot B)
\stackrel{{\scriptsize \blue  (\ref{omeg})}}{=}
\phi_{Q}\co ((i_{Q}\co g^{coH})\ot B) 
\stackrel{{\scriptsize \blue  (\ref{iphi})}}{=}
\phi_{Q}^{\omega_{Q}}\co ((i_{Q}\co g^{coH})\ot B)$$
$$\stackrel{{\scriptsize \blue  (\ref{coi-morph})}}{=}
\phi_{Q}^{\omega_{Q}}\co ((g\co i_{Q})\ot B)
\stackrel{{\scriptsize \blue  (\ref{quasilineal})}}{=}
g\co \phi_{M}^{\omega_{M}}\co (i_{M}\ot B)
\stackrel{{\scriptsize \blue  (\ref{iphi})}}{=}
g\co \phi_{M}\co (i_{M}\ot B)
\stackrel{{\scriptsize \blue  (\ref{omeg})}}{=}
g\co \beta_{M}\co n_{M^{coH}}
$$
and thus $\beta_{Q}\co g^{coH}\ot_{B^{coH}}\ot B=g\co \beta_{M}.$

{\it \underline{Step 3:}} Now we prove  the triangular identities for the unit and  counit  previously defined. Indeed:  The first triangular identity holds because composing with $n_{N}$ we have 
\begin{itemize}
\item[ ]$\hspace{0.38cm} \beta_{N\ot_{B^{coH}}B}\co (\alpha_{N}\ot_{B^{coH}}B)\co n_{N} $
\item[ ]$=\beta_{N\ot_{B^{coH}}B}\co n_{ (N\ot_{B^{coH}}B)^{coH}}\co (\alpha_{N}\ot B) $ {\scriptsize ({\blue by (\ref{mor-induc})})}
\item[ ]$= \phi_{N\ot_{B^{coH}}B}\co ((i_{N\ot_{B^{coH}}B}\co \alpha_{N})\ot B)$ {\scriptsize ({\blue by  (\ref{omeg})})}
\item[ ]$= \phi_{N\ot_{B^{coH}}B}\co ((n_{N}\co (N\ot \eta_{B}))\ot B)$ {\scriptsize ({\blue by (\ref{unit})})}
\item[ ]$= n_{N}\co (N\ot (\mu_{B}\co (\eta_{B}\ot B))) $ {\scriptsize ({\blue by (\ref{quasi-modN})})}
\item[ ]$=n_{N} $ {\scriptsize ({\blue by the unit properties})}.
\end{itemize}

Finally, if we compose with $i_{M}$,
$$i_{M}\co \beta_{M}^{coH}\co \alpha_{M^{coH}} 
\stackrel{{\scriptsize \blue  (\ref{coi-morph})}}{=}
\beta_{M}\co  i_{M^{coH}\ot_{B^{coH}}B}\co \alpha_{M^{coH}}
\stackrel{{\scriptsize \blue  (\ref{unit})}}{=}
\beta_{M}\co n_{M^{coH}}\co ( M^{coH}\ot \eta_{B})
\stackrel{{\scriptsize \blue  (\ref{omeg})}}{=}
\phi_{M}\co (i_{M}\ot \eta_{B})  
\stackrel{{\scriptsize \blue  (d2-1)}}{=}
i_{M}, $$
and then $\beta_{M}^{coH}\co \alpha_{M^{coH}} =id_{M^{coH}}.$

\end{proof}

As a consequence, we obtain the following particular instances of our main theorem.

\begin{corollary}
 The following assertions hold:
\begin{itemize}
\item[(i)] Let $H$ be a Hopf algebra and let $(B,\rho_{B})$ be a right $H$-comodule monoid such the functors $-\ot B$ and $-\ot H$ preserve coequalizers. Then, if there exists a multiplicative total integral $h:H\rightarrow B$, the categories of  right $(H,B)$-Hopf modules, denoted by ${\mathcal M}^{H}_{B}$ and introduced by Doi in \cite{Doi83}, the category ${\mathcal SM}_{B}^{H}(h)$ of strong $(H,B,h)$-Hopf modules, and the category of right $B^{coH}$-modules ${\mathcal C}_{B^{coH}}$ are equivalent. In particular, if $B=H$ and $\rho_{B}=\delta_{H}$, the Sweedler category of Hopf modules ${\mathcal M}^{H}_{H}$, the category ${\mathcal SM}_{H}^{H}(id_{H})$ of strong $(H,H,id_{H})$-Hopf modules, and the category ${\mathcal C}$ are equivalent. 
\item[(ii)] Let $H$ be a weak Hopf algebra and let $(B,\rho_{B})$ be a right $H$-comodule monoid such that the functors $-\ot B$ and $-\ot H$ preserve coequalizers. Then, if there exists a multiplicative total integral $h:H\rightarrow B$, the categories of  right $(H,B)$-Hopf modules, denoted by ${\mathcal M}^{H}_{B}$ and introduced by B\"{o}hm in \cite{bohm2} (see also \cite{ZZ} and \cite{Hanna} for the categorical equivalence), the category ${\mathcal SM}_{B}^{H}(h)$ of strong $(H,B,h)$-Hopf modules, and the category of right $B^{coH}$-modules ${\mathcal C}_{B^{coH}}$ are equivalent. In particular, if $B=H$ and $\rho_{B}=\delta_{H}$, the category of Hopf modules ${\mathcal M}^{H}_{H}$, the category ${\mathcal SM}_{H}^{H}(id_{H})$ of strong $(H,H,id_{H})$-Hopf modules, and the category ${\mathcal C}_{H_{L}}$ of right $H_{L}$-modules are equivalent. 
\item[(iii)] Let $H$ be a Hopf quasigroup and let $(B,\rho_{B})$ be a right $H$-comodule magma such the functors $-\ot B$ and $-\ot H$ preserve coequalizers. Then, if there exists an anchor morphism $h:H\rightarrow B$, the categories  ${\mathcal SM}_{B}^{H}(h)$ of strong $(H,B,h)$-Hopf modules, and the category of right $B^{coH}$-modules ${\mathcal C}_{B^{coH}}$ are equivalent. In particular, if $B=H$ and $\rho_{B}=\delta_{H}$, we obtain the result proved by Brzezi\'nski in \cite{Brz}: the category of Hopf modules ${\mathcal M}^{H}_{H}={\mathcal SM}_{H}^{H}(id_{H})$ and the category ${\mathcal C}$ are equivalent. 
\item[(iv)] Let $H$ be a weak Hopf quasigroup such the functor  $-\ot H$ preserves coequalizers. The category of strong Hopf modules ${\mathcal SM}^{H}_{H}={\mathcal SM}_{H}^{H}(id_{H})$, and the category ${\mathcal C}_{H_{L}}$ of right $H_{L}$-modules are equivalent (this is the main result  proved in \cite{Strong}). 
\end{itemize}

\end{corollary}

\begin{example}
{\rm 1.  Consider $H$  a Hopf quasigroup, $A$  a unital magma in ${\mathcal C}$, and t $\varphi_A:H\ot A\to A$ a morphism satisfying (\ref{et1}), (\ref{et2}). By the third points of Examples \ref{ex-hcm} and \ref{anch-ex} we know that  the smash product $A\sharp H$ is a right $H$-comodule magma with coaction $\rho_{A\sharp H}=A\ot \delta_{H}$ and $h=\eta_{A}\ot H:H\rightarrow A\sharp H$ is an anchor morphism. 

Moreover, if $A$ is a monoid and the equality 
\begin{equation}
\label{last-eqt}
\mu_{A}\co (\varphi_{A}\ot \varphi_{A})\co (H\ot c_{H,A}\ot A)\co (\delta_{H}\ot A\ot A)=\varphi_{A}\co (H\ot \mu_{A}), 
\end{equation}
holds, then so hold (\ref{strong-1})  and (\ref{strong2}). Indeed, first note that 

\begin{itemize}
\item[ ]$\hspace{0.38cm} q_{A\sharp H}$

\item[ ]$= A\ot (\mu_{H}\co (H\ot \lambda_{H})\co \delta_{H}$ {\scriptsize  ({\blue (\ref{nwx})})}

\item[ ]$= A\ot \eta_{H}\ot \varepsilon_{H}.$ {\scriptsize  ({\blue (\ref{rightHqg})})}.

\end{itemize}

Therefore $(A\sharp H)^{coH}=A$, $p_{A\sharp H}=A\ot \varepsilon_{H}$ and $i_{A\sharp H}=A\ot \eta_{H}$. As a consequence, we have

\begin{itemize}
\item[ ]$\hspace{0.38cm} \mu_{A\sharp H}\co ((\mu_{A\sharp H}\co (A	\ot H\ot i_{A\sharp H}))\ot A\ot H)$

\item[ ]$= (\mu_{A}\ot \mu_{H})\co (A\ot (((\mu_{A}\co (\varphi_{A}\ot \varphi_{A}))\ot H)\co (H\ot A\ot H\ot c_{H,A})\co (H\ot A\ot \delta_{H}\ot A)\co (H\ot c_{H,A}\ot A)$
\item[ ]$\hspace{0.38cm}\co (\delta_{H}\ot A\ot A)  )\ot H)$ {\scriptsize  ({\blue unit properties and associativity of $\mu_{A}$})}

\item[ ]$=(\mu_{A}\ot \mu_{H})\co   (A\ot (((\mu_{A}\co (\varphi_{A}\ot \varphi_{A})\co (H\ot c_{H,A}\ot A)\co (\delta_{H}\ot A\ot A))\ot H)\co (H\ot A\ot c_{H,A})$
\item[ ]$\hspace{0.38cm}\co (H\ot c_{H,A}\ot A)\co (\delta_{H}\ot A\ot A))\ot H)$ {\scriptsize  ({\blue coassociativity of $\delta_{H}$ and naturality of $c$ })}

\item[ ]$=(\mu_{A}\ot \mu_{H})\co   (A\ot (((\varphi_{A}\co (H\ot \mu_{A}))\ot H)\co (H\ot A\ot c_{H,A})\co (H\ot c_{H,A}\ot A)\co (\delta_{H}\ot A\ot A))\ot H)$ 
\item[ ]$\hspace{0.38cm}${\scriptsize  ({\blue (\ref{last-eqt})})}

\item[ ]$= \mu_{A\sharp H}\co (A\ot H\ot \mu_{A}\ot H)$ {\scriptsize  ({\blue naturality of $c$})}

\item[ ]$= \mu_{A\sharp H}\co (A\ot H\ot (\mu_{A\sharp H}\co (i_{A\sharp H}\ot A\ot H)))$ {\scriptsize  ({\blue (\ref{e-m-d-eps}), naturality of $c$, unit properties, and (\ref{et1})})},

\end{itemize}

and, on the other hand, 

\begin{itemize}
\item[ ]$\hspace{0.38cm} \mu_{A\sharp H}\co (i_{A\sharp H}\ot \mu_{A\sharp H})$

\item[ ]$= (\mu_{A}\ot H)\co (A\ot \mu_{A\sharp H})$ {\scriptsize  ({\blue (\ref{e-m-d-eps}), naturality of $c$, unit properties, and (\ref{et1})})}

\item[ ]$=\mu_{A\sharp H}\co  (\mu_{A}\ot  H\ot A\ot H)$ {\scriptsize  ({\blue associativity of $\mu_{A}$})}

\item[ ]$= \mu_{A\sharp H}\co ((\mu_{A\sharp H}\co (i_{A\sharp H}\ot A\ot H))\ot A\ot H)$ {\scriptsize  ({\blue (\ref{e-m-d-eps}), naturality of $c$, unit properties, and (\ref{et1})})}.

\end{itemize}

Therefore, if $-\ot A$ and $-\ot H$ preserve coequalizers, by (iii) of the previous Corollary, we have that the categories  ${\mathcal SM}_{A\sharp H}^{H}(h)$ and ${\mathcal C}_{A}$ are equivalent.

An interesting example of this case can be found using the theory developed in \cite{Majidesfera}.  Let ${\Bbb K}$ be a field and let ${\mathcal C}$ be the symmetric monoidal category of vector spaces over ${\Bbb K}$. Let $G$ be the abelian group ${\Bbb Z}_{2}^{n}$ and let $F:G\times G\rightarrow {\Bbb K}^{\ast}$ be a $2$-cochain, i.e., $F$ is a morphism such that $F(\theta, a)=F(a,\theta)=1$ for all $a\in G$ where $\theta$ is the group identity. The group algebra of $G$, denoted by ${\Bbb K}G$, is a ${\Bbb K}$-vector space with basis $\{e_{a}\;;\; a\in G\}$ and also  a unital magma with the  product (see \cite{AMajid}):
$$e_{a}e_{b}=F(a,b)e_{a+b}.$$

In what follows we will denote this magma by ${\Bbb K}_{F}G$. As was pointed in \cite{Majidesfera}, this algebraic object lives in the symmetric monoidal category of $G$-graded spaces with associator defined by the $3$-cocycle $\phi(a,b,c)=F(a,b)F(a+b,c)F(b,c)^{-1}F(a,b+c)^{-1}$  and symmetry defined by ${\mathcal R}(a,b)=F(a,b)F(b,a)^{-1}$. For example, the choice of $G={\Bbb Z}_{2}^{3}$ and certain $F$ gives the octonions. Moreover, ${\Bbb K}_{F}G$ is a composition algebra with respect to the Euclidean norm in basis $G$ if two suitable conditions hold (see (2.1) and (2.2) of \cite{Majidesfera}). This means that the norm $q(\displaystyle\sum_{a}u_{a}e_{a})=\sum_{a}u_{a}^2$ is multiplicative. Then 
$${\mathcal S}^{2^n-1}=\{\displaystyle\sum_{a}u_{a}e_{a})\;,\; \sum_{a}u_{a}^2=1_{\Bbb K}\}$$ 
is closed under the product in ${\Bbb K}_{F}G$. By Proposition 3.6 of \cite{Majidesfera} we know that ${\mathcal S}^{2^n-1}$ is an IP loop, that  becomes an usual sphere if we work over ${\Bbb R}$, and then its loop algebra, denoted by ${\Bbb K}{\mathcal S}^{2^n-1}$ is a cocommutative Hopf quasigroup (see Proposition 4.7 of \cite{Majidesfera}). Let $H$ be ${\Bbb K}{\mathcal S}^{2^n-1}$ and let $A$ be the group algebra of $G$. Then, $A$ is a monoid (it is a cocommutative Hopf algebra) and we have an action  $\varphi_{A}:H\otimes A\rightarrow A$, where $\ot =\otimes_{{\Bbb K}}$, defined by 
$$\varphi_{A}(e_{a}\otimes e_{b})=(-1)^{a.b}e_{b}.$$

It is easy to see that $\varphi_{A}$ satisfies (\ref{et1}), (\ref{et2}) and (\ref{last-eqt}) and, as a consequence of the general theory, we have a categorical equivalence between $\displaystyle{\mathcal SM}_{{\Bbb K}{\Bbb Z}_{2}^{n}\sharp {\Bbb K}{\mathcal S}^{2^n-1}}^{{\Bbb K}{\mathcal S}^{2^n-1}}(h)$ and $\displaystyle{\mathcal C}_{{\Bbb K}{\Bbb Z}_{2}^{n}}$ for $\displaystyle h=\eta_{{\Bbb K}{\Bbb Z}_{2}^{n}}\ot id_{{\Bbb K}{\mathcal S}^{2^n-1}}$.

2. By the second points of Examples \ref{ex-hcm} and \ref{anch-ex} we know that, if $H$ is a cocommutative weak Hopf quasigroup and ${\mathcal C}$ is symmetric, $(H^{op}, \rho_{H^{op}}=(H\ot \lambda_{H})\co \delta_{H})$ is an example of right $H$-comodule magma and $\lambda_{H}$ is an anchor morphism. Then, by Theorem 3.22 of \cite{Asian} and the cocommutativity of $H$, we have the equality: 
$$q_{H^{op}}=\Pi_{H}^{L}.$$
Therefore, $i_{H^{op}}=i_{L}$, $p_{H^{op}}=p_{L}$ and $(H^{op})^{coH}=H_{L}$. On the other hand, by the naturality of $c$ and (\ref{monoid-hl-2}) we obtain that 
$$\mu_{H^{op}}\co ((\mu_{H^{op}}\co (H\ot i_{L}))\ot H)=\mu_{H^{op}}\co (H\ot (\mu_{H^{op}}\co (i_{L}\ot H))),$$
and by the naturality of $c$ and (\ref{monoid-hl-3}), 
$$\mu_{H^{op}}\co (i_{L}\ot \mu_{H^{op}})=\mu_{H^{op}}\co ((\mu_{H^{op}}\co (i_{L}\ot H))\ot H).$$

Therefore, we have (\ref{strong-1}) and (\ref{strong2}). As a consequence, if the category ${\mathcal C}$ admits coequalizers and the functor $-\ot H$ preserves coequalizers, by Theorem \ref{p-paper} we obtain an equivalence between the categories ${\mathcal SM}^{H}_{H^{op}}(\lambda_{H})$ and ${\mathcal C}_{H_{L}}$. If $H$ is a Hopf quasigroup, we have a similar result that asserts the following: The categories ${\mathcal SM}^{H}_{H^{op}}(\lambda_{H})$ and ${\mathcal C}$ are equivalent.

}
\end{example}

\section*{Acknowledgements}
The authors were supported by Ministerio de Econom\'{\i}a y Competitividad (Spain), grant MTM2016-79661-P. AEI/FEDER, UE, support included (Homolog\'{\i}a, homotop\'{\i}a e invariantes categ\'oricos en grupos y \'algebras no asociativas).

\end{document}